\newif\ifmp
\mptrue 
\mpfalse

\ifmp
	\RequirePackage{fix-cm}
	\documentclass[smallextended]{svjour3}       
	\smartqed  
	\usepackage{graphicx}
\else
	\documentclass[11pt, oneside, A4paper]{article}
\fi

\usepackage{amsfonts}
\usepackage{enumerate}
\usepackage{amsmath}
\usepackage{amssymb}
\usepackage{colortbl}
\usepackage[english]{babel}
\usepackage{subfigure}
\usepackage{graphicx}
\usepackage{fancybox}
\usepackage{xcolor}
\usepackage{bbm}
\usepackage{type1cm}
\usepackage{multicol}
\usepackage{authblk}
\usepackage{algorithm}
\usepackage[noend]{algorithmic}

\ifmp
\else
	\usepackage{amsthm}
\fi

\ifmp
\else
\fi

\newcommand{\R}{\mathbb{R}}
\newcommand{\Z}{\mathbb{Z}}
\newcommand{\F}{\mathcal{F}}
\newcommand{\G}{\mathcal{G}}

\newcommand{\J}{\mathcal{J}}
\newcommand{\I}{\mathcal{I}}
\newcommand{\V}{\mathcal{V}}
\newcommand{\M}{\mathcal{M}}
\newcommand{\ones}{\mathbbm{1}}
\newcommand{\zeros}{0}
\newcommand{\supp}{\textrm{supp}}
\newcommand{\usupp}{\textrm{usupp}}

\newcommand{\zero}{{\sf zero}}

\newcommand{\Gp}{G^{\textrm{pack}}_{A, \J}}
\newcommand{\Gc}{G^{\textrm{cover}}_{A, \I}}

\newcommand{\Pp}{P^{\mathcal{V},P}}
\newcommand{\proj}{\mathrm{proj}}
\newcommand{\remove}[1]{}
\newcommand{\bigmid}{\,\middle|\,}
\newcommand{\Mod}[1]{\ (\text{mod}\ #1)}
\newcommand{\high}{{\sf high}}
\newcommand{\miss}{{\sf miss}}

\DeclareMathOperator{\argmax}{argmax}
\DeclareMathOperator{\argmin}{argmin}

\ifmp
\else
	
	\newtheorem{proposition}{Proposition}
	\newtheorem{definition}[proposition]{Definition}
	\newtheorem{example}[proposition]{Example}
	\newtheorem{lemma}[proposition]{Lemma}
	\newtheorem{theorem}[proposition]{Theorem}
	\newtheorem{corollary}[proposition]{Corollary}
	
\fi

\newtheorem{observation}[proposition]{Observation}

\newcounter{mynotes}
\newcommand{\myqed}{
\ifmp
	\hfill \qed
\fi
}

\begin{document}
\ifmp
	\title{Analysis of Sparse Cutting-planes for Sparse MILPs with Applications to Stochastic MILPs}
	\titlerunning{Analysis of Sparse Cutting-plane for Sparse MILPs} 
	
	\author{Santanu S. Dey \and Marco Molinaro \and Qianyi Wang}

	\institute{Santanu S. Dey \at School of Industrial and Systems Engineering, Georgia Institute of Technology\\
             \email{santanu.dey@isye.gatech.edu}           
             \and
             Marco Molinaro \at Computer Science Department, PUC-Rio, Brazil\\
             \email{molinaro.marco@gmail.com}
             \and
             Qianyi Wang\at School of Industrial and Systems Engineering, Georgia Institute of Technology\\
             \email{qwang32@gatech.edu}             
}

	\date{Received: date / Accepted: date}

\else
	\title{Analysis of Sparse Cutting-planes for Sparse MILPs with Applications to Stochastic MILPs}
	\author[1]{Santanu S. Dey\thanks{santanu.dey@isye.gatech.edu}}
	\author[2]{Marco Molinaro\thanks{molinaro.marco@gmail.edu}}
	\author[1]{Qianyi Wang\thanks{qwang32@gatech.edu}}
	\affil[1]{School of Industrial and Systems Engineering, Georgia Institute of Technology, Atlanta, GA 30332, United States}
	\affil[2]{Computer Science Department, PUC-Rio, Brazil}
\fi

\maketitle

\begin{abstract}
In this paper, we present an analysis of the strength of sparse cutting-planes for mixed integer linear programs (MILP) with sparse formulations.  We examine three kinds of problems: packing problems, covering problems, and more general MILPs with the only assumption that the objective function is non-negative. Given a MILP instance of one of these three types, assume that we decide on the support of cutting-planes to be used and the strongest inequalities on these supports are added to the linear programming relaxation. Call the optimal objective function value of the linear programming relaxation together with these cuts as $z^{cut}$.  We present bounds on the ratio of $z^{cut}$ and the optimal objective function value of the MILP that depends only on the sparsity structure of the constraint matrix and the support of sparse cuts selected, that is, these bounds are completely data independent. These results also shed light on the strength of scenario-specific cuts for two stage stochastic MILPs.

\ifmp
\fi

\end{abstract}

\section {Introduction}
\subsection{Motivation and goal}
Cutting-plane technology has become one of the main pillars in the edifice that is a modern state-of-the-art mixed integer linear programming (MILP) solver. Enormous theoretical advances have been made in designing many new families of cutting-planes for general MILPs (see for example, the review papers - \cite{marchand:ma:we:wo:2002,RichardDey}). The use of some of these cutting-planes has brought significant speedups in state-of-the-art MILP solvers~\cite{bixby2004,Lodi2009}.

While significant progress has been made in developing various families of cutting-planes, lesser understanding has been obtained on the question of cutting-plane selection from a theoretical perspective. Empirically, sparsity of cutting-planes is considered an important determinant in cutting-plane selection. In a recent paper~\cite{deyMolinaroWang:2015}, we presented a geometric analysis of quality of {sparse cutting-planes} as a function of the number of vertices of the integer hull, the dimension of the polytope and the level of sparsity. 

In this paper, we continue to pursue the question of understanding the strength of sparse cutting-planes using completely different techniques, so that we are also able to incorporate the information that most real-life integer programming formulations have sparse constraint matrices. Moreover, the worst-case analysis we present in this paper depends on parameters that can be determined more easily than the number of vertices of the integer hull as in~\cite{deyMolinaroWang:2015}. 	

In the following paragraphs, we discuss the main aspects of the research direction we consider in this paper, namely: (i) The fact that solvers prefer using sparse cutting-planes, (ii) the assumption that real-life integer programs have sparse constraint matrix and (iii) why the strength of sparse cutting-planes may depend on the sparsity of the constraint matrix of the IP formulation.

What is the reason for state-of-the-art solvers to bias the selection of cutting-planes towards sparser cuts? Solving a MILP involves solving many linear programs (LP) -- one at each node of the tree, and the number of nodes can easily be exponential in dimension. Because linear programming solvers can use various linear algebra routines that are able to take advantage of sparse matrices, adding dense cuts could significantly slow down the solver. In a very revealing study~\cite{Walter14}, the authors conducted the following experiment: They added a very dense valid equality constraint to other constraints in the LP relaxation at each node while solving IP instances from MIPLIB using CPLEX. This does not change the underlying polyhedron at each node, but makes the constraints dense. They observed approximately $25\%$ increase in time to solve the instances if just $9$ constraints were made artificially dense!

Is it reasonable to say that real-life integer programs have sparse constraint matrix? While this is definitely debatable (and surely ``counter examples" to this statement can be provided), consider the following statistic: the average number of non-zero entries in the constraint matrix of the instances in the MIPLIB 2010 library is $1.63\%$ and the median is $0.17\%$ (this is excluding the non-negativity or upper bound constraints). Indeed, in our limited experience, we have never seen formulations of MILPs where the matrix is very dense, for example all the variables appearing in all the constraints. Therefore, it would be fair to say that a large number of real-life MILPs will be captured by an analysis that considers only sparse constraint matrices. We  formalize later in the paper how sparsity is measured for our purposes.

Finally, why should we expect that the strength of sparse cutting-planes to be related to the sparsity of the constraint matrix of the MILP formulation? To build some intuition, consider the feasible region of the following MILP:
\begin{eqnarray*}
\begin{array}{llcl}
A^1x^1& &\leq &b^1\\
&A^2x^2&\leq& b^2\\
x^1\in \mathbb{Z}^{p_1}\times \mathbb{R}^{q_1},&x^2\in \mathbb{Z}^{p_2}\times \mathbb{R}^{q_2}&\end{array}
\end{eqnarray*}
Since the constraints are completely disjoint in the $x^1$ and $x^2$ variables, the convex hull is obtained by adding valid inequalities in the support of the first $p_1 + q_1$ variables and another set of valid inequalities for the second $p_2 + q_2$ variables. Therefore, sparse cutting-planes, in the sense that their support is not on all the variables, is sufficient to obtain the convex hull. Now one would like to extend such a observation even if the constraints are not entirely decomposable, but ``loosely decomposable". Indeed this is the hypothesis that is mentioned in the classical computational paper~\cite{crowder:jo:pa:1983}. This paper solves fairly large scale 0-1 integer programs (up to a few thousand variables) within an hour in the early 1980s, using various preprocessing techniques and the lifted knapsack cover cutting-planes within a cut-and-branch scheme. To quote from this paper:
\begin{quote}``All problems are characterized by {sparse} constraint matrix with rational data."
\end{quote}
\begin{quote}
``We note that the {support} of an inequality obtained by lifting (2.7) or (2.9) is contained in the support of the inequality (2.5) ... Therefore, the inequalities that we generate preserve the sparsity of the constraint matrix."
\end{quote}
Since the constraints matrices are sparse, most of the cuts that are used in this paper are sparse. Indeed, one way to view the results we obtain in this paper is to attempt a mathematical explanation for the empirical observations of quality of sparse cutting-planes obtained in~\cite{crowder:jo:pa:1983}. Finally, we mention here in passing that the quality of Gomory mixed integer cuts were found empirically to be related to the sparsity of LP relaxation optimal tableaux in the paper~\cite{dey:lo:wo:tr:2014}; however we do not explore particular families of sparse cutting-planes in this paper. 

\subsection{The nature of results obtained in this paper}
We examine three kinds of MILPs: packing MILPs, covering MILPs, and a more general form of MILPs where the feasible region is arbitrary together with assumptions guaranteeing that the objective function value is non-negative. For each of these problems we do the following:
\begin{enumerate}
\item We first present a method to describe the sparsity structure of the constraint matrix.
\item Then we present a method to describe a hierarchy of cutting-planes from very sparse to completely dense. The method for describing the sparsity of the constraint matrix and that for the cuts added are closely related.
\item For a given MILP instance, we assume that once the sparsity structure of the cutting-planes (i.e. the support of the cutting-planes are decided), the strongest (or equivalently all) valid inequalities on these supports are added to the linear programming relaxation and the resulting LP is solved. Call the optimal objective function value of this LP as $z^{cut}$. 
\item All our results are of the following kind: We present bounds on the ratio of $z^{cut}$ and the optimal objective function value of the IP (call this $z^I$), where the bound depends only on the sparsity structure of the constraint matrix and the support of sparse cuts. 
\end{enumerate}
For example, in the packing case, since objective function is of the maximization type, we present an upper bound on $\frac{z^{cut}}{z^I}$ which, we emphasize again, depends entirely on the location of zeros in the constraint matrix and the cuts added and is independent of the actual data of the instance. We note here that the method to describe the sparsity of the matrix and cutting-planes are different for the different types of problems.

We are also able to present examples in the case of all the three types of problems, that show that the bounds we obtain are tight. 

Though out this paper we will constantly refer back to the deterministic equivalent of a two-stage stochastic problem with finitely many realizations of uncertain parameters in the second stage. Such MILPs have naturally sparse formulations. Moreover, sparse cutting-planes, the so-called scenario-specific cuts (or the path inequalities), for such MILPs have been well studied. (See details in Section~\ref{sec:main}). Therefore, any result we obtain for quality of sparse cutting-planes for sparse IPs is applicable is this setting, and this connection allows us to shed some light on the performance of scenario-specific cuts for stochastic MILPs. 

We also conduct computational experiments for all these classes of MILPs to study the effectiveness of sparse cutting-planes. Our main observation is the sparse cuts usually perform much better than the worst-case bounds we obtain theoretically. 

Outline the paper: We present all the definitions (of how sparsity is measured, etc.) and the main theoretical results in Section \ref{sec:main}. Then in Section \ref{sec:computational} we present results from a empirical study of the same questions. We make concluding remarks in Section \ref{sec:conclusion}. Section \ref{sec:proofs} provides proofs of all the results presented in Section \ref{sec:main}.

\section{Main results}\label{sec:main}

\subsection {Notation and basic definitions} \label{sec:def}
Given a feasible region of a mixed integer linear program, say $P$, we denote the convex hull of $P$ by $P^I$ and denote the feasible region of the linear programming relaxation by $P^{LP}$. 

For any natural number $n$, we denote the set $\{1, \dots, n\}$ by $[n]$. Given a set $V$, $2^{V}$ is used to represent its power set.

\begin{definition}[Sparse cut on $N$]
Given the feasible region of a mixed integer linear program ($P$) with $n$ variables, and a subset of indices $N \subseteq [n]$, we call $\alpha^Tx \leq \beta$ \emph{a sparse cut on $N$} if it is a valid inequality for $P^I$ and the support of $\alpha$ is restricted to variables with index in $N$, that is $\left\{i \in [n]\,|\,\alpha_i \neq 0 \right\} \subseteq N$.
\end{definition}
Clarification of the above definition: If $\alpha^T x \leq \beta$ is a sparse cut on $N$, then $\alpha_i = 0$ for all $i \in [n]\setminus N$, while $\alpha_i$ may also be equal to $0$ for some $i \in N$.

Since we are interested in knowing how good of an approximation of $P^I$ is obtained by the addition of all sparse cutting-planes to the linear programming relaxation, we will study the set defined next.	

\begin{definition}[Sparse closure on $N$]\label{defn:sparseclosure}
Given a feasible region of a mixed integer linear program ($P$) with $n$ variables and $N \subseteq [n]$, we define the {sparse closure on $N$}, denoted as $P^{(N)}$, and defined as
$$P^{(N)} :=  P^{LP} \cap \bigcap_{\left\{(\alpha, \beta)\,|\, \alpha x
\leq \beta \textup{ is a sparse cut on N } \right \} } \left\{x\,|\, \alpha x \leq \beta  \right\}.$$
\end{definition}


\subsection{Packing problems}\label{results:packing}
In this section, we present our results on the quality of sparse cutting-planes for packing-type problems, that is problems of the following form:
\begin{eqnarray}
\textup{(P)}~~~\textup{max}&& c^Tx \notag \\
s.t. &&Ax \leq b \notag \\
&&x_j \in \mathbb{Z}_+, \forall j \in \mathcal{L}  \notag \\
&&x_j \in \mathbb{R}_+, \forall j \in [n]\backslash \mathcal{L}, \notag
\end{eqnarray}
with $A \in \mathbb{Q}_{+}^{m \times n}$, $b\in \mathbb{Q}_+^m$, $c \in \mathbb{Q}_+^{n}$ and $\mathcal{L} \subseteq [n]$.

In order to analyze the quality of sparse cutting-planes for packing problems we will \emph{partition the variables into blocks}. One way to think about this partition is that it allows us to understand the {global} effect of interactions between blocks of ``similar variables". For example, in MIPLIB instances, one can possibly rearrange the rows and columns~\cite{BorndorferFM98,BergnerIPCO11,WangR13,AcerKA13} so that one sees patterns of blocks of variables in the constraint matrices.  See Figure \ref{fig1}(a) for an illustration of ``observing patterns" in a sparse matrix. Moreover note that in what follows one can always define the blocks to be singletons, that is each block is just a single variable. 

The next example illustrates an important class of problems where such partitioning of variables is natural. 

\begin{example}[Two-stage stochastic problem]\label{ex:2stage}
The deterministic equivalent of a two-stage stochastic problem with finitely many realizations of uncertain parameters in the second stage has the following form:
	\begin{align*}
		\max ~~~~& c^T y + \sum_{i = 1}^k (d^i)^T z^i\\
	  s.t. ~~~~& A y \le b\\
	    	     & A^i y + B^i z^i \le b^i ~~~~~~~~\forall i \in [k],
	\end{align*}
	where $y$ are the first stage variables and the $z^i$ variables corresponding to each realization in the second stage. Notice there are two types of constraints:
\begin{enumerate}
\item Constraints involving only the first stage variables.
\item Constraints involving the first stage variables and second stage variables corresponding to one particular realization of uncertain parameters.
\end{enumerate}
Note that there are no constraints in the formulation that involve variables corresponding to two different realizations of uncertain parameters. 

It is natural to put all the first stage variables $y$ into one block and each of the second stage variables $z^i$ corresponding to one realization of uncertain parameters into a separate block of variables. 
\end{example} 


To formalize the effect of the interactions between blocks of variables we define a graph that we call as the packing interaction graph. This graph will play an instrumental role in analyzing the strength of sparse cutting-planes.

\begin{definition}[Packing interaction graph of A]\label{defn:packgraph}
Consider a matrix $A \in \mathbb{Q}^{m \times n}$. Let $\mathcal{J}:= \{J_1, J_2 \dots, J_q\}$ be a partition of the index set of columns of $A$ (that is $[n]$). We define the \emph{packing interaction graph} $\Gp = (V,E)$ as follows: 
\begin{enumerate}
\item There is one node $v_j \in V$ for every part $J_j \in \mathcal{J}$. 
\item For all $v_i, v_j \in V$, there is an edge $(v_i, v_j) \in E$ if and only if there is a row in $A$ with non-zero entries in both parts $J_i$ and $J_j$, namely there are $k \in [m]$, $u \in J_i$ and $w \in J_j$ such that $A_{ku} \neq 0$ and $A_{kw} \neq 0$.
\end{enumerate}
\end{definition}

Notice that $\Gp$ captures the sparsity pattern of the matrix $A$ up to partition $\mathcal{J}$ of columns of the matrix, i.e., this graph ignores the sparsity (or the lack of it) within each of the blocks of columns, but captures the sparsity (or the lack of it) between the blocks of the column. Finally note that if each of the blocks in $\mathcal{J}$ were singletons, then the resulting graph is the intersection graph~\cite{fulkerson:gross:1965}.

Figure \ref{fig1} illustrate the process of constructing $\Gp$. Figure \ref{fig1}(a) shows a matrix $A$, where the columns are partitioned into six variable blocks, the unshaded boxes correspond to zeros in $A$ and the shaded boxes correspond to entries in $A$ that are non-zero. Figure \ref{fig1}(b) shows $\Gp$.  

\begin{figure}[t!]
	\centering    
	\caption{Constructing $\Gp$.
	}  
	\label{fig1}
	\subfigure[The matrix $A$ with column partitions: Shaded boxes have non-zero entries.]{\label{fig:a}\includegraphics[width=50mm]{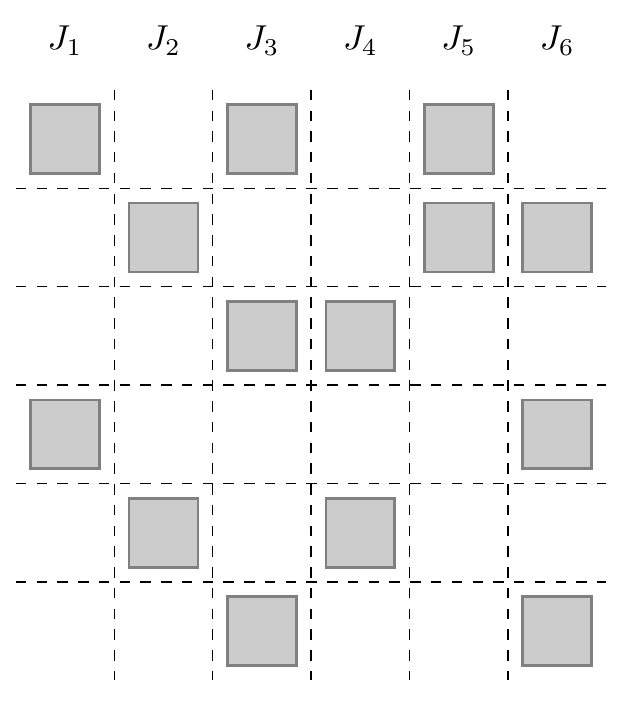}}
	\hspace{15pt}\subfigure[The resulting graph.]{\label{fig:b}\includegraphics[width=50mm]{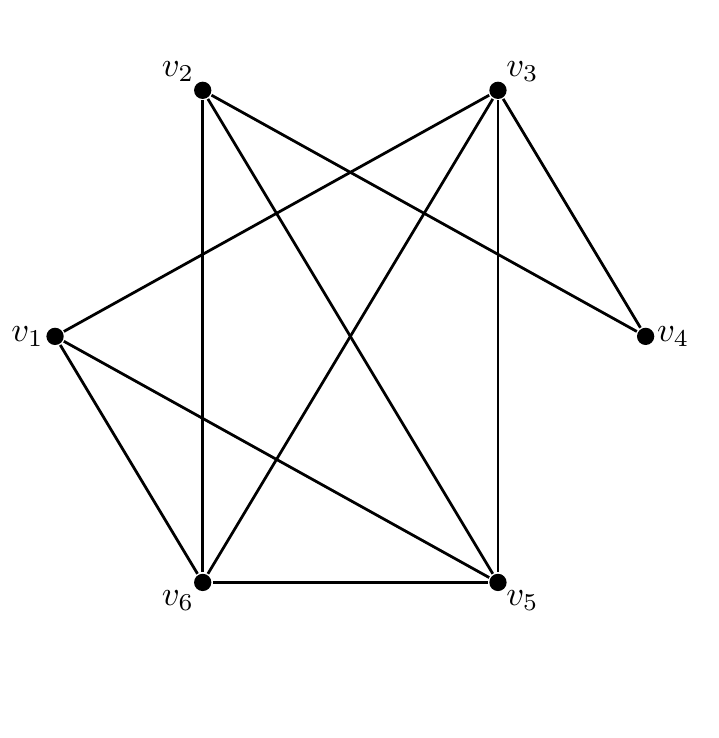}}
\end{figure}

\begin{example}[Two-stage stochastic problem: $\Gp$]\label{ex:2stage1}
Given a two-stage stochastic problem with $k$ second stage realizations, we partition the variables in $k + 1$ blocks (as discussed in Example \ref{ex:2stage}). So we have a graph $\Gp$ with vertex set $\{v_1, v_2, \dots, v_{k + 1}\}$ and edges $(v_1, v_2), (v_1, v_3), \dots, (v_1, v_{k +1})$. 
\end{example}

The sparse cuts we examine will be with respect to the blocks of variables. In other words, while the sparse cuts may be dense with respect to the variables in some blocks, it can be sparse globally if its support is on very few blocks of variables. To capture this, we use a \emph{support list} to encode which combinations of blocks cuts are allowed to be supported on; we state this in terms of subsets of nodes of the graph $\Gp$.

\begin{definition}[Column block-sparse closure]
Given the problem $\textup{(P)}$, let $\mathcal{J}:= \{J_1, J_2, \dots, J_q\}$ be a partition of the index set of columns of $A$ (that is $[n]$) and consider the packing interaction graph $\Gp = (V, E)$. 
\begin{enumerate}
	\item With slight overload in notation, for a set of nodes $S \subseteq V$ we say that inequality $\alpha x \le \beta$ is a \emph{sparse cut on} $S$ if it is a sparse cut on its corresponding variables, namely $\bigcup_{v_j \in S} J_j$. The closure of these cuts is denoted by $P^{(S)} := P^{(\bigcup_{v_j \in S} J_j)}$.
	
	\item Given a collection $\mathcal{V}$ of subsets of the vertices $V$ (the \emph{support list}), we use $\Pp$ to denote the closure obtained by adding all sparse cuts on the sets in $\mathcal{V}$'s, namely $$P^{\mathcal{V}, P} := \bigcap_{S \in \mathcal{V}} P^{(S)}.$$
\end{enumerate}
\end{definition}


This definition of column block-sparse closure allows us to define various levels of sparsity of cutting-planes that can be analyzed. In particular if $\mathcal{V}$ includes $V$, then we are considering completely dense cuts and indeed in that case $P^{\mathcal{V}, P} = P^I$.

Let $z^{I} = \textup{max}\{c^Tx \,|\, x \in P^{I}\}$ be the IP optimal value and $z^{\mathcal{V}, P} = \textup{max}\{c^Tx \,|\, x \in P^{\mathcal{V}, P}\}$ be the optimal value obtained by employing sparse cuts on the support list $\V$. Since we are working with a maximization problem $z^{\V, P} \ge z^I$ and our goal is to provide bounds on how much bigger $z^{\V, P}$ can be compared to $z^I$. Moving forward, we will be particularly interested in two types of sparse cut support lists $\mathcal{V}$:
\begin{enumerate}
\item \emph{Super sparse closure} ($P^{S. S.} := P^{\mathcal{V}, P}$ and $z^{S. S.} := z^{\mathcal{V}, P}$): We will consider the sparse cut support list $\mathcal{V} = \{ \{v_1\}, \{v_2\}, \{v_3\}, \dots, \{v_{|V|}\}$). We call this the super sparse closure, since once the partition $\mathcal{J}$ is decided, these are the sparsest cuts to be considered.  
\item \emph{Natural sparse closure}: Let $A_1, ...,A_m$ be the rows of $A$. Let $V^i$ be the set of nodes corresponding to block variables that have non-zero entries in $A_i$ (that is $V^i = \{ v_u\in V\,|\, A_{ik} \neq 0 \textrm{ for some } k \in J_u\}$).
For the resulting sparse cut support list $\mathcal{V} = \{ V^1, V^2, \dots, V^m\}$, we call the column block-sparse closure as the `natural' sparse closure (and $P^{N. S.} := P^{\mathcal{V}, P}$ and $z^{N. S.} := z^{\mathcal{V}, P}$). The reason to consider this case is that once the partition $\mathcal{J}$ is decided, the cuts defining $P^{N.S.}$ most closely resembles the sparsity pattern of the original constraint matrix. To see this, consider the case when $\mathcal{J} = \{\{1\}, \{2\}, \dots, \{n\}\}$, that is every block is a single variable. In this case, the sparse cut support list $\mathcal{V}$ represents exactly the different sparsity pattern of the various rows of the IP formulation. Indeed the cuts added in~\cite{crowder:jo:pa:1983} satisfied this sparsity pattern. 
\end{enumerate}

\begin{example}[Two-stage stochastic problem: Specific-scenario cuts, Natural sparse closure is same as relaxing ``nonanticipativity" constraints in some cases]\label{ex:2stage2}
Consider again the two-stage stochastic problem with $k$ second stage realizations as discussed in Example \ref{ex:2stage1}. Consider the cuts on the support of first stages variables together with the variables corresponding to one second stage realization, the so-called \emph{specific-scenario cuts}. Such cutting-planes are well-studied, see for example~\cite{GuanAN09,ZhangK14}. Notice that based on the partition $\mathcal{J}$ previously discussed, the closure of all the specific-scenario cuts is precisely equivalent to the natural sparse closure $P^{N.S.}$.

A standard technique in stochastic integer programming is to make multiple copies of the first stage variables, which are connected through equality constraints, and relax these (``nonanticipativity'') equality constraints (via Lagrangian relaxation methods) to produce computationally strong bound~\cite{CaroeS99}.
It is straightforward to see that in the case where there is complete recourse, the closure of the specific-scenario cuts or equivalently the natural sparse closure, 
will give the same bound as this nonanticipativity dual. 
\end{example}

To the best of our knowledge there are no known global bounds known on the quality of nonanticipativity dual. The results in this paper will be able to provide some such bounds. 

In order to present our results, we require the following generalizations of standard graph-theoretic notions such a stable sets and chromatic number.
\begin{definition}[Mixed stable set subordinate to $\mathcal{V}$]\label{defn:graph theoretic}
Let $G = (V, E)$ be a simple graph. Let $\V$ be a collection of subsets of the vertices $V$. We call a collection of subsets of vertices $\mathcal{M} \subseteq 2^{V}$ a \emph{mixed stable set subordinate to} $\mathcal{V}$ if the following hold:
\begin{enumerate}
\item Every set in $\mathcal{M}$ is contained in a set in $\V$
\item The sets in $\mathcal{M}$ are pairwise disjoint
\item There are no edges of $G$ with endpoints in distinct sets in $\mathcal{M}$.
\end{enumerate}
\end{definition}

\begin{definition}[Mixed chromatic number with respect to $\mathcal{V}$] \label{defn:chromatic}
	Consider a simple graph $G = (V,E)$ and a collection $\V$ of subset of vertices. 
	
	\begin{itemize}
		\item The \emph{mixed chromatic number} ${\bar{\eta}}^{\mathcal{V}}(G)$ of $G$ with respect to $\mathcal{V}$ is the smallest number of mixed stables sets $\M^1, \ldots, \M^k$ subordinate to $\V$ that cover all vertices of the graph (that is, every vertex $v \in V$ belongs to a set in one of the $\M^i$'s).
	
		\item (Fractional mixed chromatic number.) Given a mixed stable set $\M$ subordinate to $\V$, let $\chi_{\M} \in \{0,1\}^{|V|}$ denote its incidence vector (that is, for each vertex $v \in V$, $\chi_{\M}(v) = 1$ if $v$ belongs to a set in $\M$, and $\chi_{\M}(v) = 0$ otherwise.) 
Then we define the \emph{fractional mixed chromatic number} 
\begin{align}
	\eta^{\mathcal{V}}(G) = ~\min &\sum_{\M} y_{\M} \notag\\
\textup{s.t.}~& \sum_{\M} y_{\M} \chi_{\M} \geq \ones \label{eq:etadefn} \\
& y_{\M} \ge 0 ~~~\forall \M, \notag
\end{align}
where the summations range over all mixed stable sets subordinate to $\V$ and $\ones$ is the vector in $\mathbb{R}^{|V|}$ of all ones. 
	\end{itemize}
\end{definition}

Note that when $\mathcal{V}$ corresponds to the super sparse closure $P^{S.S.}$, that is the elements of $\mathcal{V}$ is the collection of singletons, the mixed stable sets subordinate to $\V$ are the usual stable sets in the graph and the (resp. fractional) mixed chromatic number are the usual (resp. fractional) chromatic number.

	The following simple example helps to clarify and motivate the definition of mixed stable sets: they identify sets of variables that can be set \textbf{independently} and still yield feasible solutions. 

	\begin{example}\label{ex:mixedStableSet}
		Consider the simple packing two-stage stochastic problem:
		\begin{align*}
			\max ~& c_1 x_1 + c_2 x_2 + c_3 x_3\\
			 \textrm{s.t.}~& a_{11} x_1 + a_{12} x_2 ~~~~~~~~ \le b_1\\
			 							 & a_{21} x_1 + ~~~~~~~~ a_{23} x_3 \le b_2\\
			 							 & x \in \Z^3_+.
		\end{align*}
		Consider the partition $\mathcal{J} = \{\{1\}, \{2\}, \{3\}\}$ so that the graph $\Gp$ equals the path $v_2 - v_1 - v_3$. Consider the support list $\V = \{\{1,2\}, \{1,3\}\}$ for the ``natural sparse closure'' setting. Then the maximal mixed stable sets of subordinate to $\V$ are $\M_1 = \{\{1,2\}\}$, $\M_2 = \{\{1,3\}\}$ and $\M_3 = \{\{2\}, \{3\}\}$; $\M_4 = \{\{1\}\}$ is a non-maximal mixed stable set.
		
		To see that mixed stable sets identify sets of variables that can be set independently and still yield a feasible solution, for $i =1,2,3$ let $x^{(i)}$ be the optimal solution to the above packing problem conditioned on $x_j = 0$ for all $j\neq i$; for example $x^{(2)} = (0, \lfloor b_1/a_{12} \rfloor, 0)$ and $x^{(3)} = (0, 0, \lfloor b_2/a_{23}\rfloor)$. Taking the mixed stable set $\M_3 = \{\{2\},\{3\}\}$ we see that the combination of $x^{(2)} + x^{(3)} = (0, \lfloor b_1/a_{12}\rfloor, \lfloor b_2/a_{23}\rfloor)$ is also feasible for the problem.
		
		Moreover, these solutions allow us to upper bound the ratio $z^{\V,P}/z^I$, namely the quality of the column block-sparse closure. First, the integer optimum $z^I$ is at least $\max\{c^T (x^{(2)} + x^{(3)}), c^T x^{(1)}\}$. Also, one can show that $z^{\V,P} \le c^T (x^{(2)} + x^{(3)}) + c^T x^{(1)}$ (this uses the fact that actually $x^{(2)} + x^{(3)}$ is the optimal solution for the problem conditioned on $x_1 = 0$, and $x^{(1)}$ the optimal solution conditioned on $x_2 = x_3 = 0$). Together this gives $z^{\V,P}/z^I \le 2$. Notice that the upper bound on $z^{\V}$ is obtained by adding up the solutions corresponding to the sets $\M_3$ and $\M_4$, which together cover all the variables of the problem. Looking at the fractional chromatic number $\eta^{\V}(\Gp)$ allow us to provide essentially the best such bound.
	\end{example}

Our first result gives a worst-case upper bound on $\frac{z^{\mathcal{V}, P}}{z^I}$ that is, surprisingly, independent of the data $A$, $b$, $c$, and depends only on the packing interaction graph $\Gp$ and the choice of sparse cut support list $\mathcal{V}$.


\begin{theorem}\label{thm:packing}
	Consider a packing integer program as defined in ($\textup{P}$). Let $\mathcal{J} \subseteq 2^{[n]}$ be a partition of the index
set of columns of $A$ and let $G = \Gp = (V,E)$ be the packing interaction graph of $A$. Then for any sparse cut support list $\V \subseteq 2^V$ we have
\begin{eqnarray*}
z^{\mathcal{V}, P} \leq \eta^{\mathcal{V}}(G) \cdot z^I.
\end{eqnarray*}
\end{theorem}

As discussed before, if we are considering the super sparse closure $P^{S.S.}$, $\eta^{\mathcal{V}}(G)$ is the usual fractional chromatic number. Therefore we obtain the following possibly weaker bound using Brook's theorem~\cite{brooks41}.

\begin{corollary}\label{cor:brooks}
	Consider a packing integer program as defined in ($\textup{P}$). Let $\mathcal{J} \subseteq 2^{[n]}$ be a partition of the index
set of columns of $A$ and let $\Gp$ be the packing interaction graph of $A$. Let $\Delta$ denote the maximum degree of $G$. Then we have the following bounds on the optimum value of the super sparse closure $P^{S.S}$:
\begin{enumerate}
\item If $\Gp$ is not a complete graph or an odd cycle, then
\begin{eqnarray*}
z^{S.S.} \leq \Delta \cdot z^I.
\end{eqnarray*}
\item If $\Gp$ is a complete graph or an odd cycle, then
\begin{eqnarray*}
z^{S.S.} \leq (\Delta + 1) \cdot  z^I.
\end{eqnarray*}
\end{enumerate}
\end{corollary}


Thus assuming the original IP is sparse and the maximum degree of $\Gp$ is not very high, the above result says that we get significantly tight bounds using only super sparse cuts. In fact it is easy to show the above Corollary's bounds can be tight when $\Gp$ is a 3-cycle or a star. We record this result here.

\begin{theorem}\label{thm:brookstight}
For any $\epsilon > 0$:
\begin{enumerate}
\item There exists a packing integer program as defined in ($\textup{P}$) and a partition $\J \subseteq 2^{[n]}$ of the index set of columns of $A$ such that the graph $\Gp$ is a 3-cycle and
\begin{eqnarray*}
z^{S.S.} \geq (3 - \epsilon) z^I.
\end{eqnarray*}
\item (Strength of super sparse cuts for packing two-stage problems)
There exists a packing integer program as defined in ($\textup{P}$) and a partition $\J \subseteq 2^{[n]}$ of the index set of columns of $A$ such that the graph $\Gp$ is a star and
\begin{eqnarray*}
z^{S.S.} \geq (2 - \epsilon) z^I.
\end{eqnarray*}
\end{enumerate}
\end{theorem}

We mention here in passing that there are many well-known upper bounds on the fractional chromatic number with respect to other graph properties, which also highlight that for sparse graph we expect the fractional chromatic number to be small. For example, let $G$ be a connected graph of max degree $\Delta$ and clique number $\omega(G)$. Then
\begin{enumerate}
\item $\eta(G) \leq \frac{\omega(G) + \Delta + 1}{2}$. (\cite{MolReed})
\item $\eta(G) \geq \Delta$ if and only if $G$ is a complete graph, odd cycle, a graph with $\omega(G) = \Delta$, a square of the 8-cycle, or the strong product of 5-cycle and $K_2$. Moreover if $\Delta \geq 4$ and  $G$ is not any of the graphs listed above, then $\eta(G) \leq \Delta - \frac{2}{67}$. (\cite{KingLP12})
\end{enumerate}

One question is whether we can get better bounds using the potentially denser natural sparse cuts. Equivalently, is the fractional chromatic number $\eta^{\mathcal{V}}(\Gp)$ much smaller when we consider the sparse cut support list $\mathcal{V}$ corresponding to the natural sparse closure? We prove results for some special, but important, structures.

\begin{theorem}[Natural sparse closure of trees]\label{thrm:nstree}
	Consider a packing integer program as defined in ($\textup{P}$). Let $\mathcal{J} \subseteq 2^{[n]}$ be a partition of the index
set of columns of $A$ and let $\Gp$ be the packing interaction graph of $A$. Suppose $\Gp$ a tree and let $\Delta$ be its maximum degree. Then:
\begin{eqnarray*}
z^{N.S.} \leq \left(\frac{2\Delta -1}{\Delta} \right)z^I.
\end{eqnarray*}
\end{theorem}

Compare this result for natural sparse cuts with the result for super sparse cuts on trees. While with super sparse cuts we able able to get a multiplicative bound of $2$ (this is the fractional chromatic number for bipartite graphs), using natural sparse cuts the bound is always strictly less than~$2$.

	Interestingly, this upper bound is tight even when the induced graph $\Gp$ is a star, which corresponds exactly to the case of stochastic packing programs. The construction of the tight instances are based on special set systems called \emph{affine designs}, where we exploit their particular partition and intersection properties.

\begin{theorem}[Tightness of natural sparse closure of trees] \label{thm:LBstar}
For any $\epsilon > 0$, there exists a packing integer program ($P$) and a suitable partition $\J$ of variables where $\Gp$ is a star with max degree $\Delta$ such that 
\begin{eqnarray*}
z^{N.S.} \geq \left(\frac{2\Delta-1}{\Delta} - \epsilon \right)z^I.
\end{eqnarray*}
\end{theorem}

As discussed in Example \ref{ex:2stage1}, for the case of two-stage stochastic problem with the right choice of $\mathcal{J}$ the packing interaction graph is a star. So we obtain the following corollary of Theorem \ref{thrm:nstree}.
\begin{corollary}[Strength of specific-scenario cuts for packing two-stage stochastic problems]
Consider a packing-type two-stage stochastic problem with $k$ realization. Then 
\begin{eqnarray*}
z^{N.S.} \leq \left(\frac{2k-1}{k} \right)z^I,
\end{eqnarray*}
where $z^{N.S.}$ is the objective function obtained after adding all specific-scenario cuts. Moreover this bound is tight.
\end{corollary}
We note that the analysis of approximation algorithm for two stage matching problem in the papers~\cite{EscoffierGMS2010,KongSchaefer2006} is related to the above result. We plan on exploring this relation is a future paper.

Finally we consider the case of natural sparse cutting-planes when $\Gp$ is a cycle. Interestingly, the fractional mixed chromatic number $\eta^{\mathcal{V}}(\Gp)$ depends on the length of the cycle modulo 3.  

\begin{theorem}[Natural sparse closure of cycles]\label{thm:nscycle}
	Consider a packing integer program as defined in ($\textup{P}$). Let $\mathcal{J} \subseteq 2^{[n]}$ be a partition of the index
set of columns of $A$ and let $\Gp$ be the packing interaction graph of $A$. If $\Gp$ is a cycle of length $K$, then:
\begin{enumerate}
\item If $K  = 3k, k\in \mathbb{Z}_{++}$,
 then $z^{N.S.}\leq \frac{3}{2} z^I$.
\item If $ K  = 3k +1, k\in \mathbb{Z}_{++}$, then $z^{N.S.}\leq \frac{3k+1}{2k}z^I$.
\item If $ K  = 3k +2, k\in \mathbb{Z}_{++}$, then $z^{N.S.}\leq \frac{3k+2}{2k+1}z^I$.
\end{enumerate}
Moreover, for any $\epsilon > 0$, there exists a packing integer program with a suitable partition $\mathcal{V}$  of variables, where $\Gp$ is a cycle of length $K$ such that
\begin{enumerate}
\item If $ K  = 3k, k\in \mathbb{Z}_{++}$, then $z^{N.S.}\geq \left(\frac{3}{2} - \epsilon \right)z^I$.
\item If $ K  = 3k +1, k\in \mathbb{Z}_{++}$, then $z^{N.S.}\geq \left(\frac{3k+1}{2k} - \epsilon\right)z^I$.
\item If $ K  = 3k +2, k\in \mathbb{Z}_{++}$, then $z^{N.S.}\geq \left(\frac{3k+2}{2k+1} - \epsilon\right)z^I$.
\end{enumerate}
\end{theorem}


All proofs of the above results are presented in Section \ref{sec:proofpacking}.

\subsection{Covering problems}\label{results:covering}
In this section, we present our results on the quality of sparse cutting-planes for covering-type problems, that is problems of the following form:
\begin{eqnarray}
\textup{(C)}~~~\textup{min}&& c^Tx \notag \\
s.t. &&Ax \geq b \notag \\
&&x_j \in \mathbb{Z}_+, \forall j \in \mathcal{L}  \notag \\
&&x_j \in \mathbb{R}_+, \forall j \in [n]\backslash \mathcal{L}
\notag
\end{eqnarray}
with $A \in \mathbb{Q}_+^{m \times n}$, $b\in \mathbb{Q}_+^m$, $c \in
\mathbb{Q}_+^{n}$ and $\mathcal{L} \subseteq [n]$. In this case, we would like to prove lower bounds on the objective functions after adding the sparse cutting-planes. 

Our first observation is a negative result: super sparse cuts as defined for the packing-type problems can be arbitrarily bad for the case of covering problems. In order to present this result, let formalize the notion of super sparse cuts in this setting. In particular, given an instance of type $(C)$, we assume we partition the variable indices $n$ into $\mathcal{J}  = \{J_1, J_2, \dots, J_q\}$. For all $j \in [q]$ we add all possible cuts that have support on variables with index in $J_j$. Let $z^{S.S.}$ be the optimal objective function of the resulting LP with cuts. 

\begin{theorem}\label{thm:coversupersparse}
For any constant $M >0$, there exists a covering integer program $(C)$ and partition $\mathcal{J}:= \{J_1, J_2\}$ of $[n]$, such that the corresponding $z^{S.S.}$ and $z^I$ satisfies:
	\begin{eqnarray*}
		z^I > M \cdot z^{S.S.}.
	\end{eqnarray*}
\end{theorem}

Note that super sparse cuts may have support that are strict subsets of the support on the constraints of the formulation. Theorem \ref{thm:coversupersparse} suggests that such cutting-planes in the worst case will not produce good bounds for covering problems. 

It turns out that in order to analyze sparse cutting-planes for covering problems, the interesting case is when their support is at least the support of the constraints of the original formulation. Moreover, we need to work with a graph that is a ``dual'' of $\Gp$, namely it acts on the \emph{rows} of the problem instead of columns. For the matrix $A$, let $A_i$ be the $i^{\textrm{th}}$ row. We let $\supp(A_i) \subseteq [n]$ be the set of variables which appear in the $i^{\textrm{th}}$ constraint, that is $\supp(A_i) := \{j \in [n]\,|\, A_{ij} \neq 0\}$.

\begin{definition}[Covering interaction graph of $A$] Consider the matrix $A \in \mathbb{Q}^{m \times n}$. Let $\mathcal{I} = \{I_1, I_2, \dots, I_p\}$ be a partition of index set of \textbf{rows} of $A$ (that is $[m]$). We define the \emph{covering interaction graph} $\Gc = (V, E)$ as follows:
\begin{enumerate}
\item There is a node $v_i \in V$ for every part $I_i \in \mathcal{I}$.
\item For all $v_i, v_j \in V$, there is an edge $(v_i, v_j) \in E$ if and only if there is a \textbf{column} of $A$ with non-zero entries in both parts $I_i$ and $I_j$, namely $\bigcup_{r \in I_i} \supp(A_r)$ intersects $\bigcup_{r \in I_j} \supp(A_r)$.

\end{enumerate}
\end{definition}

\begin{definition}[Row block-sparse closure]\label{defn:rowsparse}
Given the problem $\textup{(C)}$, let $\mathcal{I} = \{I_1, I_2, \dots, I_p\}$ be a partition of index set of rows of $A$ (that is $[m]$) and consider the covering interaction graph $\Gc = (V,E)$.
	\begin{enumerate}
		\item With slight overload in notation, for a set of nodes $S \subseteq V$ we say that inequality $\alpha x \le \beta$ is a \emph{sparse cut on $S$} if it is a sparse cut on the union of the support of the rows in $S$, namely $\alpha x \le \beta$ is a sparse cut on $\bigcup_{v_i \in S} \bigcup_{r \in I_i} \supp(A_r)$. The closure of these cuts is denoted by $P^{(S)} := P^{(\bigcup_{v_i \in S} \bigcup_{r \in I_i} \supp(A_r))}$.
		\item Given a collection $\V$ of subsets of the vertices $V$ (the \emph{row support list}), we use $P^{\V,C}$ to denote the closure obtained by adding all sparse cuts on the sets in $\V$'s, namely
		\begin{align*}
			P^{\mathcal{V}, C} := \bigcap_{S \in \V} P^{(S)}.
		\end{align*}  
	Moreover, we define the optimum value over the row block-sparse closure $$z^{\mathcal{V}, C} := \textup{min} \left\{c^Tx \,|\, x \in P^{\mathcal{V},C}\right\}.$$
	\end{enumerate}
\end{definition}


\begin{example}[Two-stage stochastic problem: $\Gc$, weak specific-scenario cuts]\label{ex:stoccover}
Given a two-stage covering stochastic problem with $k$ second stage realizations, we partition the rows into $k$ blocks (each block consists of constraints between first stage variables only or first stage variables and variables corresponding to one particular realization). So we have a graph $\Gc$ with $V = \{v_1, v_2, \dots, v_{k}\}$ which is a clique. Moreover, if we consider the closure corresponding to the row support list $\mathcal{V} = \{ \{v_1\}, \{v_2\}, \{v_3\}, \dots, \{v_k\}\}$, the cuts are quite similar to specific-scenario cuts (although potentially weaker, since the supports of inequalities could possibly be strictly smaller than those allowed in the ``specific-scenario cuts" in Section \ref{results:packing}). Therefore we call this closure, the weak specific-scenario closure.
\end{example}

We now present the main result of this section. In particular, we present a worst-case upper bound on $\frac{z^I}{z^{\mathcal{V}, C}}$ that is independent of the data $A$, $b$, $c$, and depends only on $\Gc$ and the choice of the row support list $\mathcal{V}$. We remind the reader that given a graph $G$ and collection $\mathcal{V}$ of its vertices, $\bar{\eta}^{\mathcal{V}}(G)$ is the mixed chromatic number with respect to $\mathcal{V}$ (see Definition~\ref{defn:graph theoretic}).

\begin{theorem}\label{thrm:covering}
	Consider a covering integer programming as defined in $(\textup{C})$. Let $\I \subseteq 2^{[m]}$ be a partition of the index set of rows of $A$ and let $G = \Gc = (V,E)$ be the covering interaction graph of $A$. Then for any sparse cut support list $\V \subseteq 2^V$ we have
	\begin{eqnarray*}
	z^{\mathcal{V}, C} \geq \frac{1}{\bar{\eta}^{\mathcal{V}}(G)}\cdot z^{I}.
	\end{eqnarray*}
\end{theorem}

We make a few comments regarding Theorem \ref{thm:coversupersparse}:
\begin{enumerate}
\item While the result of Theorem~\ref{thrm:covering} for covering-type IPs is very ``similar" to the result of Theorem~\ref{thm:packing} for packing-type of IPs, the key ideas in the proofs are different. 
\item Like the previous discussion in Section \ref{results:packing}, the chromatic numbers is small for graphs with small max degree. In fact, using Brook's Theorem~\cite{brooks41}, we can obtain a result very similar to Corollary \ref{cor:brooks} for the covering case as well. 
\item The result of Theorem \ref{thm:coversupersparse} holds even if upper bounds are present on some or all of the variables (in this case, we also need to assume that the instance is feasible). 
\end{enumerate}

Consider the case of two-stage covering stochastic problem with $K$ scenario and $\mathcal{I}$ as defined in Example \ref{ex:stoccover}. Since $\Gc$ is a clique, its chromatic number is $K$. Therefore we obtain the following corollary of Theorem \ref{thrm:covering}.

\begin{corollary}[Strength of weak specific-scenario cuts for covering stochastic problems]\label{cor:stoccover}
Consider a two-stage covering stochastic problem for $K$ scenario. Let $z^{*}$
be the optimal objective value obtained after adding all weak specific-scenario cuts. Then 
\begin{eqnarray*}
z^{*} \geq \frac{1}{K} z^{I}.
\end{eqnarray*}
\end{corollary}

Next we prove that the bound presented in Corollary \ref{cor:stoccover} is tight (and therefore the result of Theorem \ref{thm:coversupersparse} is tight for $\Gc$ being a clique).

\begin{theorem}\label{thm:stoccovertight}
Let $z^{*}$ be the optimal objective value obtained after adding all weak specific-scenario cuts for a two-stage covering stochastic problem. Given any $\epsilon >0$ with $\epsilon <K$, there exists an instance of the covering-type two-stage stochastic problem with $K$ scenarios such that
\begin{eqnarray*}
z^{*} \leq \frac{1}{(K - \epsilon)} \cdot z^{I}.
\end{eqnarray*}
\end{theorem}
The proof of Theorem \ref{thm:stoccovertight} is perhaps the most involved in this paper, as the family of instances constructed to prove the above theorem are significantly complicated. 

All proofs of the above results are presented in Section~\ref{sec:proofcovering}.


\subsection{``Packing-type" problem with arbitrary $A$ matrix}\label{results:genpacking}
Up until now we have considered packing and covering problems. We now present results under much milder assumptions. In particular, we consider problem ($\textup{P}$) with arbitrary matrix $A \in \mathbb{Q}^{m \times n}$ instead of a non-negative matrix (and $b$ is also not assumed to be non-negative). The assumptions we therefore make in this section are: $c$ is a non-negative vector, the variables are non-negative and the objective is of the maximization-type as in ($\textup{P}$).

We use the same definition of sparse-cutting planes as for the packing instances considered in Section \ref{results:packing}. All other notation used is also the same as in Section \ref{results:packing}.

As it turns out, even in this significantly more general case, it is possible to obtain tight data-independent bounds on the quality of sparse-cutting-planes. In order to present this result we introduce the notion of \emph{corrected} average constraint density. The reason to introduce this notion is the following: the strength of cuts in this case is determined by the average density, as long as the cuts cover all the variables. Based on this, the corrected average density captures the best bound one can obtain using a given support list. 

\begin{definition}[Corrected average density]\label{defn:corraverden}
Let $\mathcal{V} = \{V^1, V^2, \dots, V^t\}$ be the sparse cut support list. For any subset $\tilde{\V} =\{V^{u_1}, V^{u_2}, ..., V^{u_k}\} \subseteq \mathcal{V}$ define its \emph{density} as 
$$D(\tilde{\V}) = \frac{1}{k}\sum_{i = 1}^k|V^{u_i}|.$$
We define the corrected average density  of $\mathcal{V}$ (denoted as $D_{\mathcal{V}}$) as maximum value of $D(\tilde{\V})$ over all $\tilde{\V}$'s that cover $V$, that is, $\bigcup_{V' \in \tilde{\V}} V' = V$.
\end{definition}

Note that $D_{\mathcal{V}} \geq 1$ for any choice of sparse cut support list $\mathcal{V}$, and for the trivial list $\V = \{V(\Gp)\}$ that allows fully dense cuts we have $D_{\V} = |V(\Gp)|$. The following is the main result of this section.

\begin{theorem}\label{thm:supersparseNS}
Let ($\textup{P}$) be defined by an arbitrary $A \in \mathbb{Q}^{m \times n}$, $b \in \mathbb{Q}^m$, $c\in \mathbb{Q}^n_{+}$. Let $\mathcal{J}$ be a partition of the index set of columns of $A$ (that is $[n]$). Let $\Gp = (V,E)$ be the packing interaction graph of $A$ and let $\mathcal{V}$ be the sparse cut support list. If the instance is feasible, then:
\begin{eqnarray*}
z^{\mathcal{V}, P} \leq \left(|V| + 1 - D_{\mathcal{V}} \right) \cdot z^I.
\end{eqnarray*}
\end{theorem}

Let us see some consequences of Theorem \ref{thm:supersparseNS}. Since $D_{\V} \ge 1$ we obtain the following result. 

\begin{corollary}\label{cor:supersparseSS}
Given ($\textup{P}$), with arbitrary $A \in \mathbb{Q}^{m \times n}$, $b \in \mathbb{Q}^m$, $c\in \mathbb{Q}^n_{+}$. Let $\mathcal{J}$ be a partition of the index set of columns of $A$ (that is $[n]$). If the instance is feasible, then:
\begin{eqnarray*}
z^{\mathcal{V}, P} \leq |V| \cdot z^I.
\end{eqnarray*}
\end{corollary}

It turns out that the bound in Corollary \ref{cor:supersparseSS} is tight when $G^P_{A,\mathcal{J}}(V, E)$ is a star.

\begin{theorem}[Strength of super sparse cuts for two-stage packing-type problem with arbitrary  $A$]\label{thm:SSarbitAstartight}
For every $\epsilon >0$, there exists $A \in \mathbb{Q}^{m \times n}$, $b \in \mathbb{Q}^m$, $c\in \mathbb{Q}^n_{+}$ such that $\Gp$ is a star and
\begin{eqnarray*}
z^{S.S.} \geq |V| \cdot z^I -\epsilon.
\end{eqnarray*} 
\end{theorem}

Now let us consider the case where the sparse cut support list $\V$ corresponds to the natural sparse closure, when $G^P_{A,\mathcal{J}}(V, E)$ is a star or a cycle.
Clearly in both these cases we have $D_{\mathcal{V}} = 2$. Therefore we obtain the following Corollary.

\begin{corollary}[Natural sparse cuts for two-stage packing-type problem with arbitrary $A$]\label{cor:nsparseSS}
Given ($\textup{P}$), with arbitrary $A \in \mathbb{Q}^{m \times n}$, $b \in \mathbb{Q}^m$, $c\in \mathbb{Q}^n_{+}$. Let $\mathcal{J}$ be a partition of the index set of columns of $A$ (that is $[n]$). Let $\Gp = (V,E)$ be the packing interaction graph of $A$ which is a star or a cycle. If the instance is feasible, then:
\begin{eqnarray*}
z^{N.S.} \leq (|V| - 1) \cdot z^I.
\end{eqnarray*}
\end{corollary}

We next show that the result of Corollary \ref{cor:nsparseSS} is tight for $\Gp$ being a star, which corresponds to two-stage packing-type problem with arbitrary $A$.

\begin{theorem}\label{thm:arbitANStight}
For every $t \in \mathbb{Z}_{++}$, there exists an instance of ($\textup{P}$), with arbitrary $A \in \mathbb{Q}^{m \times n}$, $b \in \mathbb{Q}^m$, $c\in \mathbb{Q}^n_{+}$, a partition of index set of columns $\mathcal{J}$ such that  $\Gp$ is a star with $ K + 1$ nodes and
\begin{eqnarray*}
z^{N.S.} = ((K+1) - 1) \cdot z^I.
\end{eqnarray*}
\end{theorem}

All proofs of the above results are presented in Section \ref{sec:proofpackinggen}.


\section{Computational experiments}\label{sec:computational}

In this section, we present our computational results on the strength of natural sparse closure of pure binary IP. 

In Appendix $\ref{AppendixCutAlgo}$, we present the algorithm we implemented to estimate $z^{cut}$, the optimal objective function value of the natural sparse closure.

We describe the random instances we generate in section $\ref{sec:InsGen}$ and present the results in section $\ref{sec:CompRes}$. All the experiments have been carried out using CPLEX12.5.


\subsection{Instance generation}\label{sec:InsGen}
We generated two kinds of problems: two-stage stochastic programming instances and random-graph based instances. We first discuss how we generated the constraint matrix for both types of instances. Then, we discuss how we generated the right-hand side based on the constraint matrix and the objective function.

\subsubsection{Constraint matrix generation}

To simplify the presentation consider the case of packing instances. Covering instances are generated in the same way. 

First we generate the packing-type induced graph on $\texttt{nv}$ nodes. In case of the two-stage stochastic programming, the packing-type induced graph is a star with $\texttt{nv}$ nodes (i.e. $\texttt{nv} - 1$ realizations of the second stage). For the random-graph based instances, let $\texttt{p}$ (parameter) be the probability that an edge exists between any pair of nodes. As a disconnected induced graph implies that the original problem is decomposable, we accept connected graph only.

Next, given the packing-type induced graph, say $G= (V,E)$ (where $\texttt{nv} = |V|$), we construct a matrix with
that can be partitioned into $|E| \times |V|$ blocks with each block of size $\texttt{sqr} \times \texttt{sqr}$ (where $\texttt{sqr}$ is a parameter). Thus the constraint matrix has $|E| \times  \texttt{sqr}$ rows and $|V|\times \texttt{sqr}$ columns. The $(i,j)^{\textup{th}}$ block is all zeros if edge $i$ is not incident to node $j$. Else the $(i,j)^{\textup{th}}$ block is a randomly generated dense matrix: We assign each entry the distribution of $unif\{1, \texttt{M}\}$, where $\texttt{M}$ is a parameter. For packing-type with arbitrary matrix, we first generate the entry which follows $unif\{1, \texttt{M}\}$ and then with probability $0.5$, we multiply $-1$. (Thus, each column block has $\texttt{sqr}$ variables and there are $\texttt{sqr}$ rows with the same support of vertices).

\subsubsection{Right-hand side generation}
To guarantee that the instances generated are non-trivial, we follow the following steps: Randomly select $\texttt{p}_x$ (parameter) from the set of $\{0.2, 0.4,0.6,0.8\}$. A 0-1 vector $x\in \R^n$ is randomly generated where for all $j \in [n]$, $x_j \sim Bernoulli(\texttt{p}_x)$. A noise vector $\epsilon \in \R^m_{+}$ is randomly generated as: for all $i \in [m]$, $\epsilon_i \sim unif\{1, \texttt{M}_{\epsilon}\}$ ($\texttt{M}_{\epsilon}$ is a parameter). For a covering instance, $b = Ax - \epsilon$. Otherwise $b = Ax + \epsilon$.

\subsubsection{Objective function generation}

Every entry in the objective function follows the distribution of $unif\{1, \texttt{ObjM}\}$. ($\texttt{ObjM}$ is a parameter.)

\subsection{Computational results}\label{sec:CompRes}
\subsubsection{Results for two-stage stochastic programming}
We set the number of second-stage scenarios equals to $10$. We set $\texttt{sqr} = 20$, that is the number of variables for both first-stage and second-stage scenarios equals to 20. Also we set $\texttt{M} =\texttt{M}_{\epsilon} = \texttt{objM} = 10$. We generated $50$ instances for each of the three types of problem.

The result for packing-type problem, covering-type problem, and packing-type problem with arbitrary matrix is shown in Table $\ref{tab:c1}$, Table $\ref{tab:c2}$, and Table $\ref{tab:c3}$ respectively. 
\begin{table}[h]
\centering
\caption{Two-stage Packing SP}
\label{tab:c1}
\begin{tabular}{|l|l|}
\hline Avg. $z^{cut}$/$z^{IP}$ & Theoretical bound on  $z^{cut}$/$z^{IP}$ \\ \hline
1.00038 & 1.9
\\ \hline
\end{tabular}
\end{table}

\begin{table}[h]
\centering
\caption{Two-stage Covering SP}
\label{tab:c2}
\begin{tabular}{|l|l|}
\hline Avg. $z^{IP}/z^{cut}$ & Theoretical bound on  $z^{IP}/z^{cut}$ \\ \hline
1.009 & 10
\\ \hline
\end{tabular}
\end{table}

\begin{table}[h]
\centering \caption{Two-stage Arbitrary Packing SP}
\label{tab:c3}
\begin{tabular}{|l|l|}
\hline Avg. $z^{cut}$/$z^{IP}$ & Theoretical bound on  $z^{cut}$/$z^{IP}$ \\ \hline
1 & 10
\\ \hline
\end{tabular}
\end{table}
\subsubsection{Results for Random Graph based Instances}
We set $\texttt{nv} = 10$, $\texttt{p} = 0.2$, $\texttt{sqr} = 20$, $\texttt{M} =\texttt{M}_{\epsilon} = \texttt{objM} = 10$. For a given random graph we generated $10$ random instances, and therefore we generated $50$ instances for each of the three types of problem. The result for packing-type problem, covering-type problem, and packing-type problem with arbitrary matrix is shown in Table $\ref{tab:c4}$, Table $\ref{tab:c5}$, and Table \ref{tab:c6} respectively.

\begin{table}[H]
\centering \caption{Random Graph Based tests on Packing Problems} \label{tab:c4}
\begin{tabular}{|l|l|l|}
\hline Graph Name & Avg. $z^{cut}$/$z^{IP}$ & bound of $z^{cut}$/$z^{IP}$ \\
\hline Ind 1& 1.0009     & 1.8               \\
\hline Ind 2      & 1.0028 & 1.75              \\
\hline Ind 3      & 1.0053     & 1.667 \\
\hline Ind 4      & 1.0006     & 1.75              \\
\hline Ind 5      & 1.003     & 2                \\
\hline
\end{tabular}
\end{table}

\begin{table}[H]
\centering \caption{Random Graph Based tests on Covering Problems} \label{tab:c5}
\begin{tabular}{|l|l|l|}
\hline Graph Name & Avg. $z^{IP}/z^{cut}$ & bound of  $z^{IP}/z^{cut}$ \\
\hline Ind 1 & 1.0045     & 2                 \\
\hline Ind 2      & 1.0046 & 3                 \\
\hline Ind 3      & 1.0059     & 3\\
\hline Ind 4      & 1.0053     & 3                 \\
\hline Ind 5      & 1.0052     & 3                 \\
\hline
\end{tabular}
\end{table}
\begin{table}[H]
\centering \caption{Random Graph Based tests on Arbitrary Packing Problems}
\label{tab:c6}
\begin{tabular}{|l|l|l|}
\hline Graph Name & $z^{cut}$/$z^{IP}$ & bound of $z^{cut}$/$z^{IP}$ \\ \hline Ind 1
& 1     & 9               \\ \hline Ind 2      & 1
& 9             \\ \hline Ind 3      & 1    & 9
\\ \hline Ind 4      & 1     & 9              \\ \hline
Ind 5      & 1     & 9                \\ \hline
\end{tabular}
\end{table}

\section{Conclusions}\label{sec:conclusion}

In this paper, we analyzed the strength of sparse cutting-planes for sparse packing, covering and more general MILP instances. The bounds obtained are completely data independent and in particular depend only on the sparsity structure of the constraint matrix and the support of sparse cuts -- in this sense, these results truly provide insight into the strength of sparse cuts for sparse MILPs. We have shown that the theoretical bounds are tight in many cases. Especially for packing, the theoretical bounds are quite strong, showing that if we have the correct sparse cutting-planes, then the bound obtained by using these cuts may be quite good. 

The computational results are interesting: we observe that for all the types of problems sparse cutting planes perform significantly better than the theoretical prediction. This is perhaps not surprising since the theoretical bounds are data-free and therefore ``worst-case" in nature. Hence, the empirical experiments are another justification for the main message of this paper: In many cases sparse cuts provide very good bounds for sparse IPs.

\section{Proofs}\label{sec:proofs}
\subsection{Proofs for packing problems}\label{sec:proofpacking}

For any vector $x \in \mathbb{R}^n$ and $N \subseteq [n]$, we use $x|_N$ to denote the projection of $x$ on the coordinates indexed by $N$. 

	We first observe that the column sparse closure $P^{(N)}$ can be viewed essentially as the projection of $P^I$ onto the coordinates indexed by $N$.

\begin{observation}\label{prop:fund}
Consider a mixed-integer linear set with $n$ variables. For any $N \subseteq [n]$, let $P^I|_N$ be the projection of $P^I$ onto the indices in $N$. 
Then $x \in P^{(N)}$ if and only if $x \in P^{LP}$ and $x|_N \in P^I|_N$.
\end{observation}

\begin{observation}\label{obs:projPack}
		Consider a mixed-integer set of packing type and let $\mathcal{P} \subseteq \R^n$ be the set of feasible solutions. Then for any set of coordinates $N \subseteq [n]$, $x \in \R^N$ belongs to the projection $\mathcal{P}^I|_N$ iff the extension $\tilde{x} \in \R^n$ belongs to $\mathcal{P}^I$, where $\tilde{x}_i = x_i$ if $i \in N$ and $\tilde{x}_i = 0$ if $i \notin N$.

\end{observation}

%
	


\subsubsection{Proof of Theorem \ref{thm:packing}} \label{sec:packingMain}

	Recall we want to show that $z^{\V,P} \le \eta^{\V}(\Gp) \cdot z^I$. (See Example \ref{ex:mixedStableSet} for a concrete example of how the proof works.) In this section we use $P$ to denote the mixed-integer set corresponding to the packing problem $(\textup{P})$.
	
	There is a natural identification of sets of nodes of $\Gp$ with sets of indices of variables, namely if $\J = \{J_1, J_2, \ldots, J_q\}$ is the given variable index partition and the nodes of $\Gp$ are $\{v_1, v_2, \ldots, v_q\}$, then the set of vertices $\{v_i\}_{i \in I}$ corresponds to the indices $\bigcup_{i \in I} J_i \subseteq [n]$. We will make use of this correspondence, and in order to make statements precise we use the function $\phi : 2^{V(\Gp)} \rightarrow 2^{[n]}$ to denote this correspondence;
with slight abuse of notation, for a singleton set $\{v\}$ we use $\phi(v)$ instead of $\phi(\{v\})$.
	
	Given a set of vertices $S \subseteq V(\Gp)$, let $x^{(S)}$ denote the optimal solution of the packing problem conditioned on all variables $x_i$ outside $S$ taking value $0$, or more precisely, $x^{(S)} \in \argmax\{ c^T x \mid x \in P^I, x_i = 0 ~~\forall i \notin \phi(S)\}$ (we will assume without loss of generality that $x^{(S)}$ is integral); since we are working with a packing problem, this is the roughly same as optimizing over the projection of $P^I$ onto the variables in $\phi(S)$ (but notice $x^{(S)}$ lies in the original space).

	We start by showing that, roughly speaking, the closure $P^{(S)}$ captures the original packing maximization problem as long as we ignore the coordinates outside $S$.
	
	\begin{lemma}\label{lemma:pck1New}
		For any $x \in P^{(S)}$, we have $(c|_{\phi(S)})^T (x|_{\phi(S)}) \le c^T x^{(S)}$.
	\end{lemma}
	
	\begin{proof}
		Given any $x \in P^{(S)}$, Observation \ref{prop:fund} implies that $x|_{\phi(S)} \in \proj_{\phi(S)}(P^I)$. Thus, there exists points $\bar{x}^1,\ldots,\bar{x}^k \in P^I$ and $\lambda_1,\ldots, \lambda_k \in [0, 1]$, such that $x|_{\phi(S)} = \sum_{i = 1}^k \lambda_i \cdot (\bar{x}^i|_{\phi(S)})$ and $\sum _{i= 1}^k \lambda_i = 1$. Therefore, $$({c}|_{\phi(S)})^T (x|_{\phi(S)}) = \sum_{i = 1}^k \lambda_i \cdot (c|_{\phi(S)})^T (x^i|_{\phi(S)}).$$

			To upper bound the right-hand side, consider a point $x^i|_{\phi(S)}$. Let $\tilde{x} \in \R^n$ (the original space) denote the point obtained from $x^i|_{\phi(S)}$ by putting a 0 in all coordinates outside $\phi(S)$, so $\tilde{x}|_{\phi(S)} = x^i|_{\phi(S)}$ and $\tilde{x}|_{[n] \setminus \phi(S)} = \mathbf{0}$. Because $P^I$ is of packing type, notice that $\tilde{x}$ belongs to $P^I$. The optimality of $x^{(S)}$ then gives that $(c|_{\phi(S)})^T (x^i|_{\phi(S)}) = c^T \tilde{x} \le c^T x^{(S)}$.
		
		Employing this upper bound on the last displayed equation gives  $$({c}|_{\phi(S)})^T (x|_{\phi(S)}) \le \sum_{i = 1}^k \lambda_i \cdot c^T x^{(S)} = c^T x^{(S)},$$ thus concluding the proof. \myqed
	\end{proof}
	
	Now we lower bound the packing problem optimum $z^I$ by solutions constructed via mixed stable sets.
	
\begin{lemma}\label{lem:addstableset}
Given any mixed stable set $\mathcal{M}$ for $\Gp$, the point $\sum_{M \in \M} x^{(M)}$ belongs to $P^I$. Thus, $z^I \ge \sum_{M \in \mathcal{M}} c^T x^{(M)}$.
\end{lemma}

\begin{proof}
	We just prove the first statement. First, notice since each $x^{(M)}$ is integral and non-negative, so is the point $\sum_{M \in \M} x^{(M)}$. So consider an inequality $A_i x \le b_i$ in $(\textup{P})$. For any two sets $M_1 \neq M_2 \in \M$, notice that the vector $A_i$ either has all zeros on the indices corresponding to $M_1$ or on the indices corresponding to $M_2$, namely either $A_i|_{\phi(M_1)} = \textbf{0}$ or $A_i|_{\phi(M_2)} = \textbf{0}$. Applying this to all pairs of sets in $\M$, we get that there is only one set $M^* \in \M$ such that $A_i|_{\phi(M^*)}$ is non-zero, which implies that 
	\begin{align*}
		A_i \sum_{M \in \M} x^{(M)} = \sum_{M \in \M} A_i x^{(M)} = \sum_{M \in \M} (A_i|_{\phi(M)}) (x^{(M)}|_{\phi(M)}) = A_i x^{(M^*)} \le b_i,
	\end{align*}
	where the last inequality follows from the feasibility of $x^{(M^*)}$. Thus, the point $\sum_{M \in \M} x^{(M)}$ satisfies all inequalities $A_i x \le b_i$ of the system $(\textup{P})$, concluding the proof. \myqed
	
\end{proof}
	

Now, we present the proof of Theorem \ref{thm:packing}.

\begin{proof}[Proof of Theorem \ref{thm:packing}]
		Let $V = \{v_1, \ldots, v_q\}$ denote the vertices of $\Gp$, and let $\sf{MSS}$ denote the set of all mixed stable sets of $\Gp$ with respect to $\mathcal{V}$.	Let $\{y_M\}_{M \in \sf{MSS}}$ be an optimal solution of linear problem (\ref{eq:etadefn}) corresponding to the definition of mixed fractional chromatic number with respect to $\mathcal{V}$, and define $g = \sum_{\M \in \sf{MSS}} \chi_\M y_\M \in \R^{q}$. Based on the constraints of \eqref{eq:etadefn} we have that $g \geq \ones$.
		
		We upper bound the optimum $z^{\V,P}$ of the column block-sparse closure. For that let $$x^* = \argmax \{c^Tx\,|\,x \in P^{\mathcal{V}, P}\},$$ be the optimal solution. Then breaking up the indices of the variables based on the nodes $V$ and using the non-negativity of $c$ and $x^*$, we have
	\begin{align*}
		z^{\mathcal{V}, P} &= c^T x^* = \sum_{j = 1}^q (c|_{\phi(v_j)})^T (x^*|_{\phi(v_j)}) \leq \sum_{j=1}^{q} g_j \cdot (c|_{\phi(v_j)})^T (x^*|_{\phi(v_j)})  \\
		& =  \sum_{j=1}^{q} \left(\sum_{\M \in \sf{MSS}} (\chi_\M)_j \cdot  y_\M \right) \cdot (c|_{\phi(v_j)})^T (x^*|_{\phi(v_j)})\\
		&=\sum_{\M \in \sf{MSS}} y_\M \cdot \left(\sum_{j=1}^{q} (\chi_\M)_j \cdot (c|_{\phi(v_j)})^T (x^*|_{\phi(v_j)})\right)\\
		&= \sum_{\M \in \sf{MSS}} y_\M \cdot \left(\sum_{\{j\,|\,v_j \in \M\}} (c|_{\phi(v_j)})^T (x^*|_{\phi(v_j)})\right)\\
		&=  \sum_{\M \in \sf{MSS}} y_\M \cdot \left(\sum_{S \in \M} (c|_{\phi(S)})^T (x^*|_{\phi(S)})\right).
	\end{align*}
	
	To further upper bound the right-hand side consider some $\M \in \sf{MSS}$, some $S \in \M$ and the term $(c|_{\phi(S)})^T (x^*|_{\phi(S)})$. First we claim that $x^*$ belongs to the column block-sparse closure $P^{(S)}$. To see this, first recall from the definition of mixed stable set that there must be a set $V_S$ in the support list $\V$ containing $S$. Moreover, since $x^* \in P^{\V,P} = \bigcap_{V' \in \V} P^{(V')}$, we have $x^* \in P^{(V_S)}$; finally, the monotonicity of closures implies $P^{(S)} \supseteq P^{(V_S)}$, and hence $x^* \in P^{(S)}$. Thus we can employ Lemma \ref{lemma:pck1New} to obtain the upper bound $(c|_{\phi(S)})^T (x^*|_{\phi(S)}) \le c^T x^{(S)}$.
	
	Plugging this bound on last displayed inequality and using Lemma  \ref{lem:addstableset} we then get
	\begin{align*}
		z^{\V,P} \le \sum_{\M \in \sf{MSS}} y_\M \cdot \left(\sum_{S \in \mathcal{M}} c^T x^{(S)} \right) \leq \sum_{\M \in \sf{MSS}} y_\M \cdot z^I = \eta^{\mathcal{V}}(\Gp) \cdot z^I.
	\end{align*}
	This concludes the proof. \myqed
	\end{proof}

\subsubsection{Proof of Corollary~\ref{cor:brooks} and Theorem~\ref{thm:brookstight}}

Brooks' Theorem~\cite{brooks41} is the following result (recall that a proper coloring of a graph is an assignment of colors to the vertices such that no edge has the same color on both endpoints).

\begin{theorem}[Brook's Theorem]
Consider a connected graph $G$ of max degree $\Delta$. Then $G$ can be properly colored by $\Delta$ colors, except in two cases either when $G$ is a complete graph or an odd cycle, in which case it can be properly colored with $\Delta + 1$ colors.
\end{theorem}

Since the fractional chromatic number is a lower bound on the chromatic number, we obtain Corollary \ref{cor:brooks}. 

We now prove Theorem~\ref{thm:brookstight}.

\begin{proof}[Proof of Theorem \ref{thm:brookstight}] 
Recall that super sparse closure corresponds to the support list $$\mathcal{V} = \{\{v_1\}, \{v_2\}, \dots, \{v_{|V|}\}\}.$$
The proof of both parts is similar. 

	\paragraph{Part 1.} 
	We want to show an example where $z^{SS} \geq (3 - \epsilon) z^I$ for all $\epsilon >0$ where $\Gp$ is a 3-cycle. We will construct an integer program with $3$ variables and $\mathcal{J} = \left\{\{1\}, \{2\}, \{3\}\right\}$. Given $\epsilon > 0$, consider the following packing integer program:
\begin{eqnarray}
\begin{array}{rcccl}
\textup{max}& x_1 &+ x_2 &+ x_3& \\
\textup{s.t.}& x_1 &+ x_2& &\leq 2 - \frac{2}{3}\epsilon \\
&x_1 & &+ x_3 &\leq 2 - \frac{2}{3}\epsilon \\
&&x_2 &+ x_3& \leq 2 - \frac{2}{3}\epsilon \\
&x_1 \in \mathbb{Z}_{+},&x_2 \in \mathbb{Z}_{+},& x_3 \in \mathbb{Z}_{+} 
\end{array}
\end{eqnarray}
Clearly $\Gp$ is a 3-cycle. 
Note that the only valid inequalities that have support on each of the three blocks defined by $\J$ is $x_i \leq 1$ for $i =1,2,3$. Thus, the point $(1 - \frac{\epsilon}{3}, 1 - \frac{\epsilon}{3}, 1 - \frac{\epsilon}{3})$ belongs to the super sparse closure $P^{S.S.}$, and hence the optimum value satisfies $z^{S.S.} \geq 3 - \epsilon$. On the other hand clearly, $z^{IP} = 1$, concluding the proof.

	\paragraph{Part 2.} 
	We want to show an example where $z^{SS} \geq (2 - \epsilon) z^I$ for all $\epsilon >0$ where $\Gp$ is a star.
	Take $\Delta \in \mathbb{Z}_{+}$. We construct a packing integer program with $2\Delta$ variables and $\mathcal{J} = \left\{\{1, 2, \dots, \Delta\}, \{\Delta + 1\}, \{\Delta + 2\}, \dots, \{2\Delta\}\right\}$. Given $\epsilon > 0$, consider the following integer program
\begin{eqnarray*}
\textup{max}& \sum_{i = 1}^{2\Delta} x_i\\
\textup{s.t.}& x_i + x_{\Delta + i} \leq 2 - \epsilon \ \forall i \in [\Delta]\\
&x \in \mathbb{Z}^{2\Delta}_{+}.
\end{eqnarray*}
Clearly $\Gp$ is a star with $\Delta$ leaves. Letting $P$ be the associated mixed integer set of the above integer program, note that the projection of $P^{I}$ to the first block of variables $\{1, 2, \ldots, \Delta\}$ equals $[0, 1]^{\Delta}$. Also the only valid inequalities that have support on each of the other $\Delta$ blocks $\{\Delta + i\}$ is $0 \leq x_i \leq 1$ for $i \in \{\Delta + 1, \dots, 2\Delta\}$. Thus, the point $x$ with $x_i = 1$ for all $i \in [\Delta]$ and $x_i = 1- \epsilon$ for all $i \in \{\Delta + 1, \dots, 2\Delta\}$ belongs to the super sparse closure $P^{S.S.}$. Thus the optimum $z^{S.S.}$ is at least $2\Delta - \Delta\epsilon$. On the other hand, clearly $z^{I} = \Delta$, concluding the proof.\myqed
\end{proof}


\subsubsection{Proof of Theorem~\ref{thrm:nstree}}

 We prove the desired upper bound $z^{N.S.} \leq \left(\frac{2\Delta-1}{\Delta} \right) \cdot z^I$. Due to Theorem \ref{thm:packing}, it suffices to upper bound the fractional chromatic number $\eta^{\V}(\Gp) \le \frac{2\Delta-1}{\Delta}$ for $\V$ set according to the natural sparse closure setting. Notice however, that in this setting every edge of $E = E(\Gp)$ belongs to some set in $\V$ and vice-verse, and therefore $\eta^{\V}(\Gp) = \eta^E(\Gp)$. Thus, it suffices to prove $\eta^E(\Gp) \le \frac{2\Delta-1}{\Delta}$.

\remove{
\begin{proposition}\label{prop:treeupbnd}
Let ($\textup{P}$) be defined by $A \in \mathbb{Q}_+^{m \times n}$, $b \in \mathbb{Q}^m$, $c\in \mathbb{R}^n_{+}$ and $\mathcal{L} \subseteq
[n]$. Let $\mathcal{J}:= \{J_1, J_2 \dots, J_q\}$ be a partition of the index set of columns of $A$ (that is $[n]$). Let $G^P_{A, \mathcal{J}}(V, E)$ be the packing interaction graph of $A$. Suppose $G^P_{A, \mathcal{J}}(V,E)$  a tree and let $k$ be the maximum node degree of $G^P_{A, \mathcal{J}}(V,E)$. 
Then:
\begin{eqnarray*}
z^{N.S.} \leq \left(\frac{2k-1}{k} \right)z^I.
\end{eqnarray*}
\end{proposition}

\begin{proof}
}

	The following is the main tool for providing an efficient mixed stable set fractional coloring. 

	\begin{lemma}\label{lemma:treeMSS}
		Let $T = (V, E)$ be a tree of maximum degree $\Delta$. Then there is a collection of $2\Delta-1$ sets of edges $E_1, E_2, \ldots, E_{2\Delta-1}$ and $2\Delta-1$ sets of nodes $V_1, V_2, \ldots, V_{2\Delta -1}$ satisfying the following:
		\begin{enumerate}
			\item For each $i \in [2\Delta -1]$ the collection $E_i \cup V_i$ is a mixed stable set for $T$ subordinate to $E$
			
			
			\item Each node of $T$ is covered exactly $\Delta$ times by the collection of mixed stable sets $\{E_i \cup V_i\}_{i \in [2\Delta - 1]}$.
		\end{enumerate}
	\end{lemma}
	
	\begin{proof}
		If $T$ consists of a single edge, then $\Delta = 1$ and we can simply set $E_1$ to be the edge of $T$ and set $V_1 = \emptyset$ to get the desired sets. So assume that $T$ has at least one internal node.
		
		In order to simplify the proof we make all the degrees the same: construct the tree $T'$ from $T$ by adding new leaves to all internal nodes of $T$ so that now every internal node of $T'$ has degree exactly $\Delta$. We will construct the desired sets $\{E'_i\}_{i \in [2\Delta - 1]}$ and $\{V'_i\}_{i \in [2\Delta - 1]}$ for $T'$ via a coloring argument reminiscent of the proof of Brook's Theorem (although not the same argument).
				
		 Pick any internal node $v_0$ of $T'$ and root this tree at $v_0$. We label all edges and leaf nodes of $T'$ with numbers in $[2\Delta -1]$ according to the following BFS procedure (we use the standard meaning of ``parent'', ``child'', ``depth'' (where $v_0$ has depth 0, $L$ is the maximum dept of any node), etc. for rooted trees):

\begin{algorithm}[ht!]
\begin{algorithmic}
\label{dalg:convex-sub}
	\STATE Label each of the $\Delta$ edges incident to the root $v_0$ with a distinct label
	
	\FOR{$i = 1$ to $L$}
		\FOR{every vertex $v$ of depth $i$}
			\STATE Let $S$ denote the set of labels assigned to all the edges incident to the parent of $v$ and notice that $|S| = \Delta$
			
			\IF{$v$ is an internal node}
				\STATE Label the $\Delta-1$ edges of $v$ to its children with distinct labels from the set $[2\Delta - 1] \setminus S$
			\ELSE
				\STATE Assign \textbf{all} $\Delta-1$ labels $[2\Delta - 1] \setminus S$ to $v$.
			\ENDIF
		\ENDFOR
	\ENDFOR
\end{algorithmic}
\end{algorithm}
	Then for all $j \in [2 \Delta -1]$, let $E'_j$ (resp. $V'_j$) be set of edges (resp. nodes) of $T'$ that have label $j$ (notice that vertices have multiple labels). 
	
	It follows directly from the labeling procedure that each set $E'_i \cup V'_i$ is a mixed stable set of $T'$ (and clearly subordinate to the edges of $T'$). Now to see that each node $v$ of $T'$ is covered exactly $\Delta$ times by the collection $\{E'_i \cup V'_i\}_{i \in [2\Delta - 1]}$ we consider 2 cases: If $v$ is an internal node, then by construction of $T'$ it has degree exactly $\Delta$ and since $\bigcup_i E'_i = E(T')$ it is covered $\Delta$ times by the collection $\{E'_i\}_i$ and 0 times by the collection $\{V'_i\}_i$, giving the desired result. On the other hand, if $v$ is a leaf of $T'$, then it is covered once by the set $E'_i$ where $i$ is the label of the only edge incident to $v$, covered by no other set $E'_j$, and covered by the $\Delta-1$ sets $V'_j$ corresponding to the labels of $v$. Thus, the sets $\{E'_i\}_i$ and $\{V'_i\}_i$ satisfy the desired properties with respect to the modified tree $T'$.
	
	Now to get the desired sets for the original tree $T$, we just remove the nodes in $T'\setminus T$ from the sets $E'_i \cup V'_i$: for each $\{v, v'\}$ with $v \in V(T)$ and $v' \notin V(T)$ that belongs to $E'_i \cup V'_i$, replace it with the singleton $\{v\}$; denote the set obtained by $E_i \cup V_i$ (concretely, $E_i$ is the set of pairs in this collection and $V_i$ is the set of singletons in this collection). Notice that this replacement procedure does not add repeated singletons: this is because $E'_i \cup V'_i$ contains only disjoint edges (the labeling scheme above does not assign color $i$ to two intersecting edges) and if it contains an edge $(v,v')$ with $v' \notin V(T)$ then this implies that $v$ is an internal node of $T$ and hence the singleton $\{v\}$ does not belong to $E'_i \cup V'_i$.
	
	It follows directly from this replacement operation that the sets $E_i \cup V_i$'s are mixed stable sets for $T$ subordinate to $E(T)$ and that still each node of $T$ is covered exactly $\Delta$ times by them. This concludes the proof. \myqed
	\end{proof}

	The upper bound in Theorem \ref{thrm:nstree} the follows from the following corollary.

	\begin{corollary}
		Let $H$ be a tree of maximum degree $\Delta$. Then $\eta^E(H) \le \frac{2 \Delta -1}{\Delta}$.
	\end{corollary}
	
	\begin{proof}
		Consider the mixed stable sets $\M_i = E_i \cup V_i$ (for $i \in [2\Delta -1]$) of $H$ obtained from Lemma \ref{lemma:treeMSS}. Since each vertex of $H$ is covered exactly $\Delta$ times by $\{\M_i\}_{i \in 2\Delta - 1}$, we have that setting $y_{\M_i} = \frac{1}{\Delta}$ for all $i$ (and $y_\M = 0$ otherwise) yields a feasible solution for the mixed fractional chromatic number program \eqref{eq:etadefn} of value $\frac{2\Delta -1}{\Delta}$, proving the result.\myqed
	\end{proof}


	\subsubsection{Proof of Theorem \ref{thm:LBstar}}


	Fix $\epsilon > 0$; we construct an instance where $z^{N.S.} \ge \left(\frac{2\Delta-1}{\Delta} - \epsilon \right) \cdot z^I$. The construction require the existence of the so-called \emph{affine designs}.

\begin{definition} Given $n \in \mathbb{Z}_{++}$, we call an \emph{affine $n$-design} a
collection $\mathcal{F}_1, \dots,  \mathcal{F}_n$ where each
$\mathcal{F}_i$ is a family of $n$-subsets of $[n^2]$ satisfying:
\begin{enumerate}
\item For any $i \in [n]$, the sets in $\F_i$ partition $[n^2]$
\item For any $i \neq j \in [n]$ and $A \in \mathcal{F}_i$ and $B \in
\mathcal{F}_j$, we have $|A \cap B| \leq 1$.
\end{enumerate}
\end{definition}

\begin{theorem}[\cite{colbourndi:2006}, Part VII, Point 2.17] For every prime $n$, an affine n-design exists.
\end{theorem}

	So consider a prime number $n \ge \Delta$ and let $\F_1, \ldots, \F_n$ be an affine $n$-design. For a set $A \in \F_i$ we use $\chi_A \in \{0,1\}^{n^2}$ to denote the indicator vector of the set $A$.
	
	We will construct a packing IP in $\mathbb{R}_+^{n^2+\Delta}$ and partition the $n^2 + \Delta$  variables into $\Delta +1$ blocks $\mathcal{J} = \{J_0,\ldots, J_{\Delta}\}$ by setting $J_0 = \{1,\ldots, n^2\}$ and $J_i = \{n^2+i\}$, for $i = 1,\ldots, \Delta$. To simplify the notation we use $x \in \mathbb{R}^{n^2}$ to represent the variables
in $J_0$, and $y_i$, $i = 1, \ldots, \Delta$, to represent the variables in $J_i$ respectively. 

	Let $P_i$ be the polytope in $\R^{n^2 + \Delta}$ given by the convex hull of the points 
	\begin{align*}
		\left\{(x, y_1, \ldots, y_\Delta) \in \{0,1\}^{n^2 + \Delta} \bigmid \sum_j x_j \le n \textrm{ and, } \textrm{if $y_i = 1$, then } x \le \chi_A \textrm{ for some } A \in \F_i \right\};
	\end{align*}
explicitly, this is the set of solutions satisfying 
	\begin{align}	
		&\sum_j x_j \le n\notag\\
		&x_a + x_b + y_i \le 2 ~~~~\forall a \in A, b \in B, ~A \neq B \in \F_i \label{eq:designDown}\\
		&(x, y_1, \ldots, y_\Delta) \in [0,1]^{n^2 + \Delta}.\notag
	\end{align}
	Then the desired IP $(\textup{P}$) is obtained by considering the integer solutions common to all these polytopes:
	\begin{align*}
		\max &\sum_{j \in [n^2]} x_j + \left(\frac{n-1}{\Delta-1}\right) \sum_{j \in [\Delta]} y_j\\
		& (x, y_1, \ldots, y_{\Delta}) \in \bigcap_{i \in [\Delta]} P_i   \cap \Z^{n^2 + \Delta}.
	\end{align*}
	Again we get the following interpretation for the feasible solutions for this problem: in any solution $(x, y) \in \{0,1\}^{n^2 + \Delta}$, $\sum_j x_j \le n$ and for all $i \in [\Delta]$
	\begin{align}
		\textrm{if $y_i= 1$, then $x \le \chi_A$ for some set $A \in \F_i$}\label{eq:tightDesign}
	\end{align}
	
	Let $P$ denote the integer set corresponding to this problem.
	From the explicit description of the $P_i$'s we see that this is packing integer program whose induced graph $\Gp$ is a star with maximum degree $\Delta$.
	
	The intuition behind this construction is the following: first, maximizing the objective function over just $P_i \cap \Z^{n^2 + \Delta}$ (or equivalently over $P_i$) gives value $n + \left(\frac{n - 1}{\Delta - 1}\right) \cdot \Delta \approx 2n$ (by taking $x = \chi_A$ for any $A \in \F_i$, $y_j = 1$ for all $j$). Moreover, recall that the natural sparse closure w.r.t. $\J$ of the full program $(\textup{P})$ uses cuts that are only supported in $(x, y_i)$, for $i \in [\Delta]$; thus, roughly speaking, this closure sees each $P_i \cap \Z^{n^2 + \Delta}$ independently, and not really capturing the fact they are being intersected. Thus, optimizing over the natural sparse closure w.r.t. $\J$ still gives value $\approx 2n$. However, due to the fact the sets across the design's $\F_i$'s are almost disjoint, intersecting the regions $P_i \cap \Z^{n^2 + \Delta}$ kills most of the solutions. A bit more precisely, the almost disjointness in the affine design and expression \eqref{eq:tightDesign} imply that the best solution either sets many of the $y_i$'s to 1 and almost all $x_j$'s to 0, or sets all $x_j$'s to 1 and few $y_i$'s to 0; these solutions gives value $\approx n$. This gives the desired gap of $\approx 2$ between the natural sparse closure and the original IP.
	
	To make this formal, we start with the following lemma.
	
	\begin{lemma} \label{lemma:designNS}
		Setting $x = (\frac{1}{n}, \ldots, \frac{1}{n})$ and $y_j = 1$ for all $j$ gives a feasible solution to the natural sparse closure $P^{N.S}$. Thus, $z^{N.S.} \ge n + \left(\frac{n-1}{\Delta-1}\right) \cdot \Delta$.
	\end{lemma}
	
	\begin{proof}
		Let $\bar{x} = (\frac{1}{n}, \ldots, \frac{1}{n})$ and $\bar{y} = (1, \ldots, 1)$ denote the desired solution.
		
		We claim that it suffices to prove that $(\bar{x}, e^i)$ belongs to $P^I$ for all $i$, (where $e^i$ is the $i$th canonical basis vector in $\R^{\Delta}$). To see that, first notice that the natural sparse closure w.r.t. $\J$ is 	$P^{N.S.} = \bigcap_{i \in [\Delta]} P^{(x,y_i)}$, where we use $P^{(x,y_i)}$ to denote the sparse closure of $P$ with cuts on variables $(x,y_i)$ (see Definition~\ref{defn:sparseclosure}). 
Using Observations \ref{prop:fund} and \ref{obs:projPack}, it suffices to show $(\bar{x}, \bar{y}) \in P^{LP}$ and $(\bar{x}, e^i) \in P^I$. The former condition can be easily verified via equation \eqref{eq:designDown}, so it suffices to show $(\bar{x}, e^i) \in P^I$ for all $i$.
		
		So fix $i \in [\Delta]$. Consider the collection $\F_i$ and a point of the form $(\chi_A, e^i)$ for any set $A \in \F_i$. By definition of $P_i$, notice that $(\chi_A, e^i)$ belongs to $P_i \cap \Z^{n^2 + \Delta}$. Moreover, notice that for $j \neq i$ we also have $(\chi_A, e^i) \in P_j$: this follows from the facts $\sum_j (\chi_A)_j \le n$ and $e^i_j = 0$. Thus, we have $(\chi_A, e^i) \in P^I = \bigcap_{j \in [\Delta]} P_j \cap \Z^{n^2 + \Delta}$. Then the average $\sum_{A \in \F_i} \frac{1}{n} (\chi_A, e^i)$ belongs to $P^I$; since the sets in $\F_i$ form a partition of $[n^2]$, $\sum_{A \in \F_i} \chi_A = (1,\ldots, 1)$, and hence the average is $\sum_{A \in \F_i} \frac{1}{n} (\chi_A, e^i) = (\bar{x}, e^i) \in P^I$. This concludes the proof. \myqed
	\end{proof}

	The next step is to understand $P$ better. 
	
	\begin{lemma}\label{lemma:designIP}
		For any solution $(x, y_1, \ldots, y_\Delta) \in P$ with $\sum_{i \in [\Delta]} y_i \ge 2$ we have $\sum_{j \in [n^2]} x_j \le 1$.
	\end{lemma}
	
	\begin{proof}
		Consider $p\neq q$ such that $y_p = y_q = 1$. By definition of $P$, we have that the solution $(x, y_1, \ldots, y_\Delta)$ belongs to $P_p$ and $P_q$. Since $y_p = y_q = 1$, this means that there are sets $A \in \F_p$ and $B \in \F_q$ such that $x \le \chi_A$ and $x \le \chi_B$, which further implies $x \le \chi_{A \cap B}$. But by definition of an affine design $|A \cap B| \le 1$, and hence $\sum_{j \in [\Delta]} x_j \le 1$. This concludes the proof. \myqed
	\end{proof}

	\begin{corollary}\label{cor:designOpt}
		We have $z^I \le \frac{n \Delta - 1}{\Delta -1}$.
	\end{corollary}
	
	\begin{proof}
		Consider any feasible solution $(x, y_1, \ldots, y_\Delta)$ to $P^I$. From the first constraint in \eqref{eq:designDown} we have $\sum_j x_j \le n$. Thus, if $\sum_i y_i \le 1$, the solution has value at most $n + \left(\frac{n-1}{\Delta -1}\right) = \frac{n \Delta - 1}{\Delta - 1}$; on the other hand, using Lemma \ref{lemma:designIP}, if $\sum_i y_i \ge 2$ then the solution has value at most $1 + \left(\frac{n-1}{\Delta -1}\right) \Delta = \frac{n \Delta - 1}{\Delta -1}$. Together these give the desired upper bound.\myqed
	\end{proof}
	
	Lemma \ref{lemma:designNS} and Corollary \ref{cor:designOpt} give that
	\begin{align*}
		\frac{z^{N.S.}}{z^I} \geq
\left( n + \left(\frac{n-1}{\Delta -1}\right) \cdot \Delta\right) \frac{\Delta-1}{n \Delta - 1} = \frac{2n \Delta - n - \Delta}{n \Delta-1} = \frac{2 \Delta - 1 - \Delta/n}{\Delta - 1/n}.
	\end{align*}
	Since $\lim_{n \rightarrow \infty} \frac{2 \Delta - 1 - \Delta/n}{\Delta - 1/n} = \frac{2\Delta -1}{\Delta}$, for a sufficiently large choice of $n$ we get $z^{N.S.} \ge \left(\frac{2 \Delta - 1}{\Delta} - \epsilon\right) z^I$. This concludes the proof of Theorem \ref{thm:LBstar}.

\subsubsection{Proof of the first part of Theorem~\ref{thm:nscycle}: upper bound on $z^{N.S.}$}


	Consider the packing interaction graph $\Gp$, which is a cycle of length $K$. Notice that the natural sparse closure in this case corresponds to considering the support list $\V$ being simply the edges of $\Gp$. Thus, to prove the first part of Theorem \ref{thm:nscycle} is suffices to upper bound the fractional mixed chromatic number $\eta^{E(\Gp)}(\Gp)$.
	
	We can work more abstractly to simplify things: let $H = (V,E)$ be the cycle $v_0 - v_1 - \ldots - v_{K-1} - v_0$ on $K$ nodes, and we need to upper bound $\eta^E(H)$. To further simplify the notation, we identify $v_i$ with $v_{i \Mod{K}}$ for $i \ge K$. We consider the different cases depending on $K \Mod{3}$.

\remove{
Given the set of nodes $V = \{v_1, \ldots , v_K\}$ of $G^P_{A, \mathcal{J}}(V, E)$, we use the notation that for $i > K$, $v_{i} := v_{i (\textup{mod} \ K)}$. The set of edges $E = \{e_1, \ldots, e_K\}$ where $e_i = (v_i, v_{i+1})$. Since we are studying natural sparse closure, we have that $\mathcal{V}: = \bigcup_{i = 1}^K\{v_i,  v_{i +1}\} $. We will prove the following statements.
\begin{enumerate}
\item If $ K  = 3k, k\in \mathbb{Z}_{++}$, $\eta^{\mathcal{V}}_{(G^P_{A, \mathcal{J}})}\leq \frac{3}{2}$.
\item If $ K  = 3k +1, k\in \mathbb{Z}_{++}$, $\eta^{\mathcal{V}}_{(G^P_{A, \mathcal{J}})}\leq \frac{3k+1}{2k}$.
\item If $ K  = 3k +2, k\in \mathbb{Z}_{++}$, $\eta^{\mathcal{V}}_{(G^P_{A, \mathcal{J}})}\leq \frac{3k+2}{2k+1}$.
\end{enumerate}
}

\paragraph{Case 1: $K = 3k$, $k \in \mathbb{Z}_{++}$.} For $i = 0,1,2$, let $\M_i$ denote the set of edges $\{v_j, v_{j+1}\}$ where $j = i \Mod{3}$. It is clear that each $\M_i$ is a mixed stable set for $H$ subordinate to $E$. Moreover, since $\bigcup_{i = 0}^3 \M_i = E$ covers each node of $H$ exactly twice, we can find a solution for the fractional mixed chromatic number LP \eqref{eq:etadefn} by setting $y_{\M_i} = \frac{1}{2}$ for $i=0,1,2$. This gives the desired bound $\eta^E(H) \le \frac{3}{2}$.

\begin{figure}[h!]
	\centering    
	\caption{Constructions of all mixed stable sets for 6-cycle}
	\label{fig6c}
	\subfigure{\includegraphics[width=120mm]{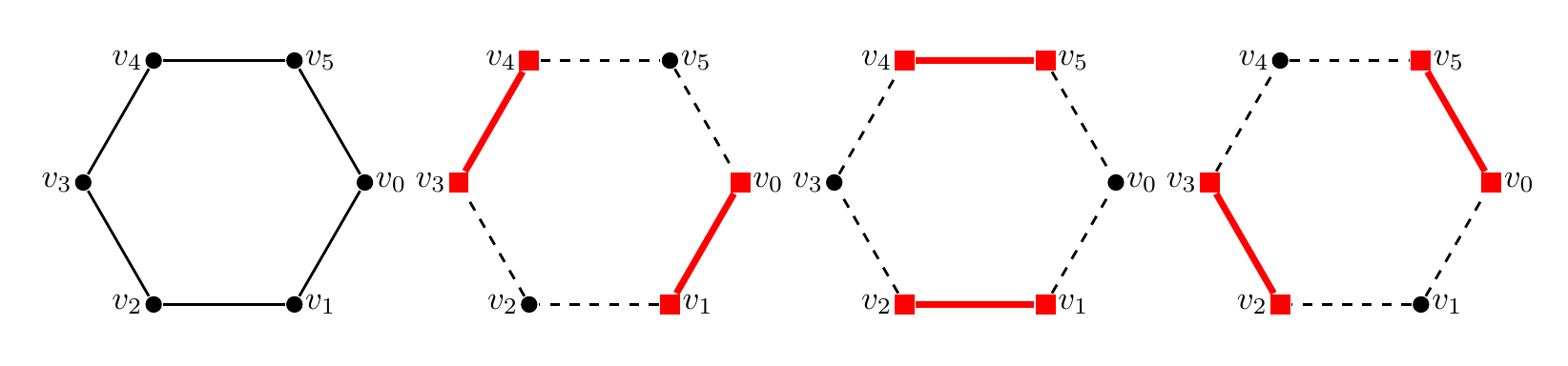}}
\end{figure}
\paragraph{Case 2: $K = 3k + 1$, $k \in \mathbb{Z}_{++}$.} We show $\eta^E(H) \le \frac{3k + 1}{2k}$. If $k = 1$, we have that $H$ is a 4-cycle and define $\M_0 = \{(v_0,v_1)\}$, $\M_1 = \{(v_1, v_2)\}$, $\M_2 = \{(v_2, v_3\})$ and $\M_3 = \{(v_3, v_0)\}$. Clearly these $\M_i$'s are mixed stable sets for $H$ subordinate to $E$ and $\bigcup_i \M_i$ covers each node of $H$ exactly twice; then as in the previous case, this gives $\eta^E(H) \le \frac{4}{2} = 2 = \frac{3k + 1}{2k}$.

	For $k \ge 2$, define
\begin{eqnarray*}
\mathcal{M}_{i} =  \left\{ \{v_{i}, v_{i+1}\}, \{v_{i + 3}, v_{i + 4}\}, \ldots, \{v_{3(k-2)+i}, v_{3(k-2)+i + 1}\}, \{v_{3(k-1)+i}, v_{3(k-1)+i + 1}\}\right\}
\end{eqnarray*}
for $i = \{0, \dots, 3k\}$. 
It is straightforward to check that each $\mathcal{M}_i$ is a mixed stable set subordinate to $E$ and that $\bigcup_{i = 0}^{3k} \mathcal{M}_i$ covers every node exactly $2k$ times. Thus again we get $\eta^E(H) \le \frac{3k + 1}{2k}$.

\begin{figure}[h!]
	\centering    
	\caption{Constructions of all mixed stable sets for 7-cycle}
	\label{fig7c}
	\subfigure{\includegraphics[width=120mm]{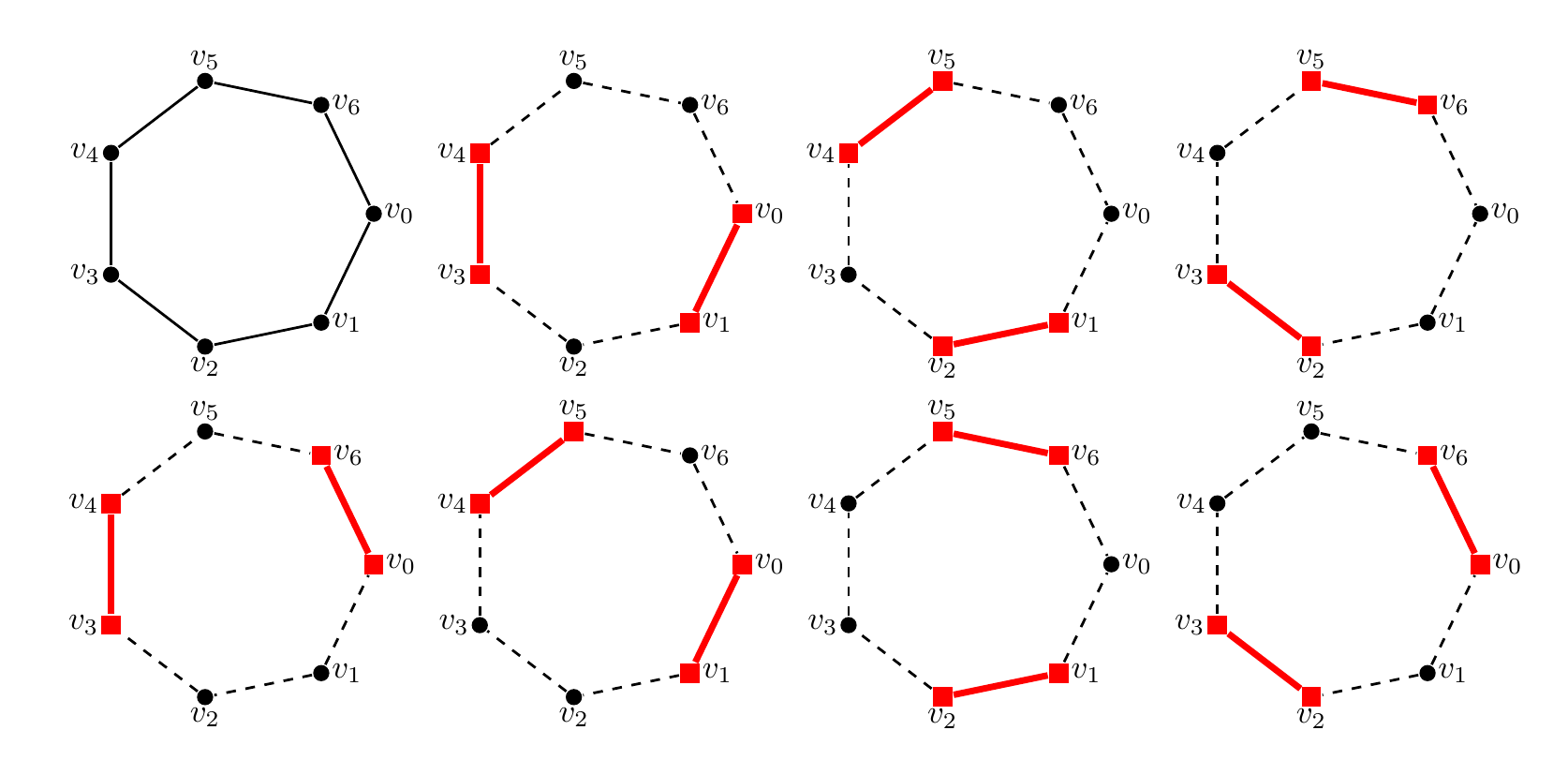}}
\end{figure}
\paragraph{Case 3: $K = 3k + 2$, $k \in \mathbb{Z}_{++}$.} Let
\begin{eqnarray*}
\mathcal{M}_{i} =  \left\{ \{v_{i}, v_{i+1}\}, \{v_{i + 3}, v_{i + 4}\}, \ldots, \{v_{3(k-2)+i}, v_{3(k-2)+i + 1}\}, \{v_{3(k-1)+i}, v_{3(k-1)+i + 1}\}, \{v_{3k + i}\} \right\}
\end{eqnarray*}
for $i = \{0, \dots, 3k + 1\}$. It is straightforward to check that each $\mathcal{M}_i$ is a mixed stable set subordinate to $E$ and that $\bigcup_{i = 0}^{3k+1} \mathcal{M}_i$ covers every node exactly $2k + 1$ times. Thus we have $\eta^E(H) \le \frac{3k+2}{2k + 1}$. This concludes the proof of the first part of the theorem.

\begin{figure}[h!]
	\centering    
	\caption{Constructions of all mixed stable sets for 5-cycle}
	\label{fig5c}
	\subfigure{\includegraphics[width=90mm]{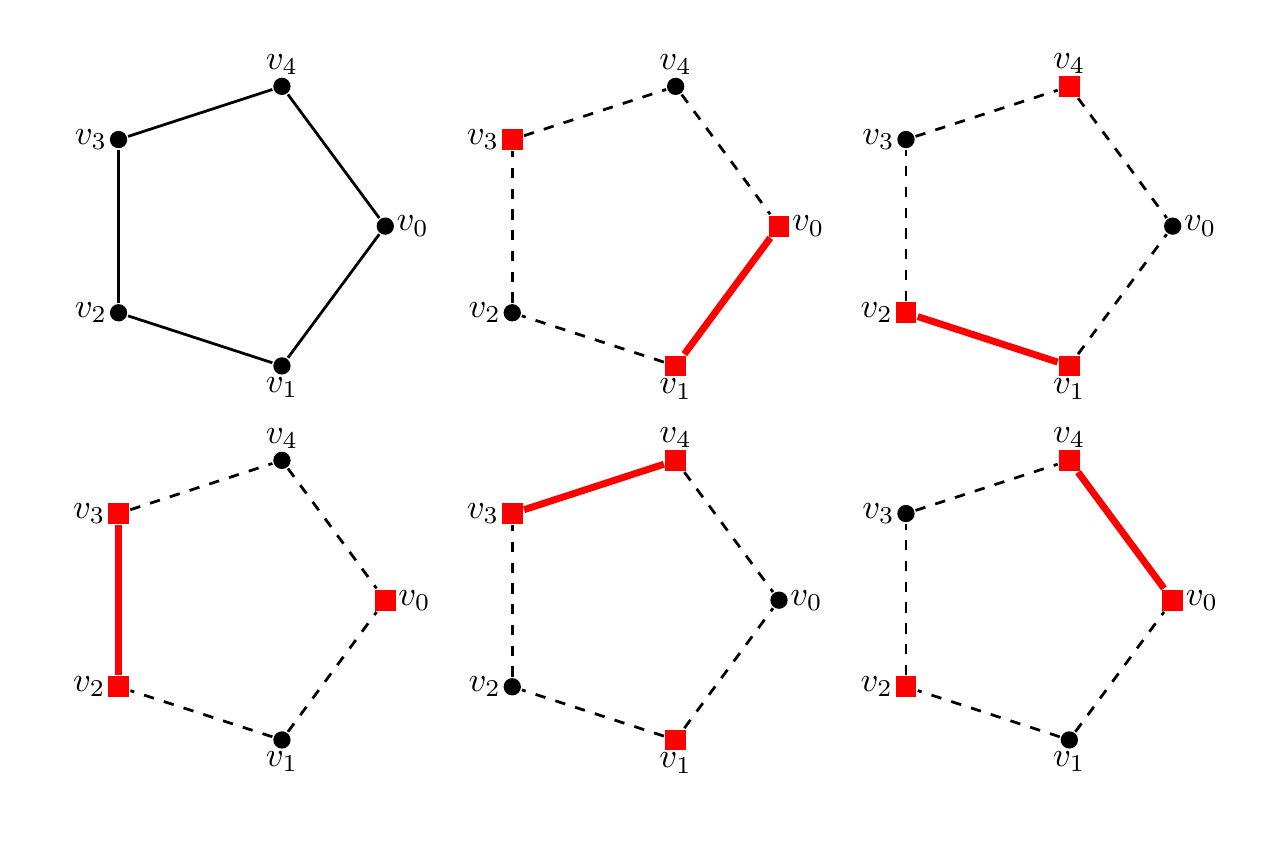}}
\end{figure}

\subsubsection{Proof of second part of Theorem \ref{thm:nscycle}: tight instances}

	The construction of the tight instances is similar to the one used in Theorem \ref{thm:LBstar}. 
	So consider a prime number $n \ge K$ and let $\F_1, \ldots, \F_n$ be an affine $n$-design. For a set $A \in \F_i$ again we use $\chi_A \in \{0,1\}^{n^2}$ to denote the indicator vector of the set $A$.

	We will construct a packing IP with $K n^2$ variables, which are partitioned into $K$ equally sized blocks $\mathcal{J} = \{J_0,\ldots, J_{K - 1}\}$, namely $J_i = \{n^2 i, n^2 i + 1, \ldots, n^2 i + n^2 - 1\}$. To simplify the notation, we use $x^i \in \mathbb{R}^{n^2}$ to represent the variables corresponding to $J_i$, so a solution of the IP has the form $(x^0, \ldots, x^{K-1})$. For $i \ge K$, we use $x^i$ to denote $x^{i \Mod{K}}$.

	First, define the integer set $Q = \{x \in \{0,1\}^{n^2} \mid \ones^T x \le n\}$. Then, for $i \in \{0, \ldots, K-1\}$ let $P_i$ be the polytope in $\R^{K n^2}$ given by the convex hull of the points 
	\begin{align*}
		\left\{(x^0, \ldots, x^{K-1}) \in Q^K \bigmid 
			\begin{array}{l}
				\textrm{if $x^i \neq 0$, then } x^{i+1} \le \chi_A \textrm{ for some } A \in \F_i \textrm{, and}\\
				\textrm{if $x^{i + 1} \neq 0$, then } x^i \le \chi_A \textrm{ for some } A \in \F_i
			\end{array}\right\};
	\end{align*}
explicitly, this is the set of solutions satisfying 
	\begin{align}	
		&\ones^T x^j \le n ~~~~~~~~~~~~~~~~~~~~\forall j\notag\\
		&x^i_a + x^i_b + x^{i+1}_c \le 2 ~~~~~~~~\forall a \in A, b \in B, ~A \neq B \in \F_i, ~\forall c\label{eq:designDown2}\\
		&x^{i+1}_a + x^{i+1}_b + x^i_c \le 2 ~~~~~~\forall a \in A, b \in B, ~A \neq B \in \F_i, ~\forall c \notag\\
		&(x^0, \ldots, x^{K-1}) \in [0,1]^{K n^2}.\notag
	\end{align}


	Then the desired IP $(\textup{P}$) is obtained by considering the integer solutions common to all these polytopes:
	\begin{align*}
		\max &\sum_{i = 0}^{K-1} \ones^T x^i \\
		& (x^0, \ldots, x^{K-1}) \in \bigcap_{i = 0}^{K-1} P_i \cap \Z^{K n^2}.
	\end{align*}
	Again we get the following interpretation for the feasible solutions for this problem: in any solution $(x^0, \ldots, x^{K-1}) \in \{0,1\}^{K n^2}$, $\ones^T x^i \le n$ for all $i$, and also for all $i$
	\begin{align}
		&\textrm{if $x^i \neq 0$, then $x^{i+1} \le \chi_A$ for some set $A \in \F_i$, and}\label{eq:tightDesign2} \\
		&\textrm{if $x^{i+1} \neq 0$, then $x^i \le \chi_A$ for some set $A \in \F_i$}. \notag
	\end{align}
	
	Let $P$ denote the integer set corresponding to this problem. From the explicit description of the $P_i$'s we see that this is packing integer program whose induced graph $\Gp$ is a $K$-cycle.

	We now consider the natural sparse closure $P^{N.S.}$ and the integer hull $P^I$
for this problem and lower bound the ratio $z^{N.S.}/z^I$. For that, given $x = (x^0, \ldots, x^{K-1})$, let $\high(x) = \{ i \mid \ones^T x^i \ge 2\}$, namely the set of block of variables with ``high'' value. We say that three integers are \emph{adjacent mod $K$} if they are of the form $i \Mod{K}, i+1 \Mod{K}, i+2 \Mod{K}$. 
	
	\begin{lemma}\label{lemma:adjacent}
		For any solution $x \in P$, the set $\high(x)$ does not contain any three adjacent mod $K$ integers. 
	\end{lemma}
	
	\begin{proof}
		By contradiction, assume that $\high(x)$ contains the integers $i \Mod{K}, i + 1 \Mod{K}$, and $i + 2 \Mod{K}$. In particular, all of $x^i$, $x^{i+1}$ and $x^{i+2}$ are different from 0, and hence expression \eqref{eq:tightDesign2} implies that $x^{i+1} \le \chi_A$ and $x^{i+1} \le \chi_B$ for some $A \in \F_{i \Mod{K}}$ and $B \in \F_{i + 1 \Mod{K}}$; this implies that $x^{i+1} \le \chi_{A \cap B}$. But by definition of affine design, we have $|A \cap B| \le 1$, and hence $\ones^T x^{i+1} \le 1$, reaching a contradiction.\myqed
	\end{proof}

	The following lemma can be easily checked. 
	
	\begin{lemma} \label{lemma:sizeNonAdj}
		Let $S$ be a subset of $\{0, \ldots, K-1\}$ that does not contain any three adjacent mod $K$ integers. Then: (i) if $K = 3k$ or $K = 3k + 1$ for $k \in \Z_{++}$ we have $|S| \le 2k$, and; (ii) if $3k + 2$ for $k \in \Z_{++}$ we have $|S| \le 2k + 1$. 
	\end{lemma}	

\remove{
	\begin{proof}
		Suppose $K = 3k$. Then we have $|\{0, 1, 2\} \bigcap S| \leq 2$, $|\{3, 4, 5\} \bigcap S| \leq 2$, $\dots, |\{K - 3, K - 2, K - 1\} \bigcap S| \leq 2$. Therefore the result holds.

		Now suppose $K \neq 3k$. If $K = 4$ the result can be easily checked. For $K \geq 5$ one of the two cases can occur:

\begin{enumerate}
\item  $\{i, i + 1\} \not\subseteq T^K(\hat{x})$ for all $i$: In this case, $|T^K(\hat{x})| \leq \left \lfloor \frac{K}{2}\right \rfloor$, which is smaller than the bound. 
\item There exists $i \in \{1, \dots, K\}$ such that $\{i, i + 1\} \subseteq T^K(\hat{x})$: WLOG, we may assume that $i = 1$. By Claim 1 we have that $3 \notin T^K(\hat{x})$ and $K \notin T^K(\hat{x})$. Then 
\begin{enumerate}
\item If $K = 3k + 1$, then since $|\{4, 5, 6\} \bigcap T^K(\hat{x})| \leq 2$, $|\{7, 8, 9\} \bigcap T^K(\hat{x})| \leq 2$, $\dots, |\{3k - 2, 3k - 1, 3k \} \bigcap T^K(\hat{x})| \leq 2$, we obtain that $|T^K(\hat{x})| \leq 2 + 2(k -1) = 2k$.
\item If $K = 3k + 2$, then since $|\{4, 5, 6\} \bigcap T^K(\hat{x})| \leq 2$, $|\{7, 8, 9\} \bigcap T^K(\hat{x})| \leq 2$, $\dots, |\{3k - 2, 3k - 1, 3k \} \bigcap T^K(\hat{x})| \leq 2$, we obtain that $|T^K(\hat{x})| \leq 2 + 2(k -1) + 1 = 2k  + 1$.~
\end{enumerate}
\end{enumerate}\myqed
	\end{proof}
}

	\begin{lemma} \label{lemma:LBzICycle}
		The optimal value of the integer program $(\textup{P})$ can be upper bounded as follows: if $K = 3k$ or $K = 3k + 1$ for $k \in \Z_{++}$, $z^I \le (n - 1) \cdot 2k + K$; if $K = 3k + 2$ for $k \in \Z_{++}$, $z^I \le (n - 1) \cdot (2k + 1) + K$.
	\end{lemma}
	
	\begin{proof}
		Let $\bar{x} = (\bar{x}^0, \ldots, \bar{x}^{K-1})$ be an optimal solution to $(\textup{P})$. Using the fact that $\ones^T \bar{x}^i \le n$ and the definition of $\high(\bar{x})$ we get
	\begin{align*}
	z^I = \sum_{i= 0}^{K-1} \ones^T \bar{x}^i &= \sum_{i \in \high(\bar{x})} \ones^T \bar{x}^i + \sum_{i \notin \high(\bar{x})} \ones^T \bar{x}^i \\
	&\le n \cdot |\high(\bar{x})| + K - |\high(\bar{x})| = (n-1) \cdot |\high(\bar{x})| + K.
	\end{align*}	
	Upper bounding $|\high(\bar{x})|$ using Lemmas \ref{lemma:adjacent} and \ref{lemma:sizeNonAdj} gives the desired result.\myqed
	\end{proof}

	\begin{lemma} \label{lemma:LBNS2}
		The point $\bar{x} = (\frac{1}{n} \ones, \ldots, \frac{1}{n} \ones)$ is a feasible solution to the natural sparse closure $P^{N.S}$. Thus, $z^{N.S.} \ge K n$.
	\end{lemma}
	
	\begin{proof}	
		To simplify the notation, let $\zero^i(x,x') \in \R^{n^2} \times \ldots \R^{n^2}$ denote the vector $$(0, \ldots, 0, x, x', 0, \ldots, 0)$$ where $x$ is in the $i$th position and $x'$ is in position $i + 1 \Mod{K}$.
		 
		We claim that it suffices to prove that $\zero^i(\ones/n, \ones/n)$ belongs to $P^I$ for all $i$. To see that, first notice that the natural sparse closure w.r.t. $\J$ is 	$P^{N.S.} = \bigcap_{i \in 0}^{K-1} P^{(x^i,x^{i+1})}$, where we use $P^{(x^i,x^{i+1})}$ to denote the sparse closure of $P$ on variables $(x^i,x^{i+1})$.
Using Observations \ref{prop:fund} and \ref{obs:projPack}, it suffices to show $\bar{x} \in P^{LP}$ and $\zero^i(\ones/n, \ones/n) \in P^I$. The former condition can be easily verified via equation \eqref{eq:designDown}, so it suffices to show $\zero^i(\ones/n,\ones/n) \in P^I$.
				
		So fix $i$. Consider the collection $\F_i$. By the definition of $P_i$, for each $A,B \in \F_i$ the point $\zero^i(\chi_A, \chi_B)$ belongs to $P_i \cap \Z^{K n^2}$. If also follows directly from the definition of $P_j$ that $\zero^i(\chi_A, \chi_B) \in P_j$ for all $j \neq i$. Thus, we have $\zero^i(\chi_A, \chi_B) \in P^I = \bigcap_{j = 0}^{K-1} P_j \cap \Z^{n^2 + \Delta}$. Then the following average belongs to $P^I$:
		\begin{align*}
			\sum_{A \in \F_i} \frac{1}{n} \sum_{B \in \F_i} \frac{1}{n} \zero^i(\chi_A, \chi_B) =	\sum_{A \in \F_i} \frac{1}{n} \zero^i\left(\chi_A, \sum_{B \in \F_i} \frac{1}{n} \chi_B\right) = \zero^i\left(\sum_{A \in \F_i} \frac{1}{n}, \sum_{B \in \F_i} \frac{1}{n}  \right).
		\end{align*}	
		Recalling that $\sum_{A \in \F_i} \chi_A = \ones$, this average is $\zero^i\left(\frac{1}{n} \ones, \frac{1}{n} \ones\right) \in P^I$. This concludes the proof.\myqed
	\end{proof}

	Putting Lemmas \ref{lemma:LBzICycle} and \ref{lemma:LBNS2} together, we get that if $K = 3k$ for $k \in \Z_{++}$, $\frac{z^{N.S.}}{z^I} \ge \frac{Kn}{(n-1) \cdot 2k + K} = \frac{n}{(n-1) \cdot (2/3) + 1} = \frac{1}{2/3 + 1/3n}$. Since $\lim_{n \rightarrow \infty} \frac{1}{2/3 + 1/3n} = \frac{3}{2}$, for sufficiently large $n$ we have $z^{N.S.} \ge z^I (\frac{3}{2} - \epsilon)$, proving this part of the theorem. The other cases of $K \Mod{3}$ are similar. This concludes the proof. 
 
\subsection{Proof for covering problem}\label{sec:proofcovering}

\subsubsection{Proof of Theorem~\ref{thm:coversupersparse}}
In order to prove Theorem \ref{thm:coversupersparse}, we begin with a classical bad example for the LP relaxation of the set cover problem.

\begin{definition}[Special set covering problem (SSC)]
	Consider $q \in \mathbb{Z}_+$. The ground set of the set cover problem will be $\{0,1\}^q$, and the covering sets $S(v) = \{ u \in \{0,1\}^q \setminus \{\zeros\} \mid v^T u = 1 \Mod{2}\}$ for $v \in \{0,1\}^q$. Then the Special Set Covering (SSC(q)) problem is defined by:
	\begin{align*}
		(SSC(q))~~\min &\sum_{v \in \{0,1\}^q} x_v  \\
		s.t.& \sum_{v : u \in S(v)} x_v \geq 1 ~~~~~\forall u \in \{0,1\}^q\\
		& x_v \in \{0,1\}  ~~~~~~~~~~\forall v \in \{0,1\}^q.
	\end{align*}
We refer to the left-hand matrix of SSC as $A^q$.
\end{definition}


\begin{theorem}[\cite{lovasz:1975}]\label{thm:knownresult}
The IP optimal value of $SSC(q)$ is at least $q$, while the LP relaxation optimal value is at most $2$. 
\end{theorem}

	We are now ready to present the proof of Theorem~\ref{thm:coversupersparse}. For that, we consider the following problem:
	\begin{align*}
		(DSC(q))~~\min& \sum_{v \in \{0,1\}^q} x_v +  \sum_{v \in \{0,1\}^q} y_v\\
		s.t. ~~&A^q x + A^q y \geq \ones\\
		& x,y \in \{0,1\}^{2^q}
	\end{align*}

We first argue that $DSC(q)$ preserves the gap between IP and LP from $SSC(q)$.

\begin{lemma}
	The IP optimal value of $DSC(q)$ is at least $q$, while the LP relaxation optimal value is at most $2$. 	
\end{lemma}

\begin{proof}
	$(z^{LP} \le 2)$: By Theorem~\ref{thm:knownresult}, there exists a feasible solution $\bar{x}$ of the LP relaxation of $SSC(q)$ such that $\sum_v \bar{x}_v \leq 2$. Then $(\bar{x}, 0)$ is a feasible solution of the LP relaxation of $DSC(q)$, giving the desired bound.
	
	\smallskip \noindent $(z^{IP} \ge q)$: Assume by contradiction that $(\bar{x}, \bar{y})$ is a feasible solution of $DSC(q)$ with objective function less than $q$. Note that if $\bar{x}_v = \bar{y}_v = 1$, then we may set $\bar{y}_v = 0$ and still obtain a feasible solution with a better objective function value; similarly, if $\bar{x}_v = 0$ and $\bar{y}_v = 1$ we may set $\bar{x}_v = 1$ and $\bar{y}_v = 0$ and obtain a feasible solution with same objective value. Therefore, we may assume that $\bar{y} = 0$. In this case, $\bar{x}$ is a feasible solution of SSC with objective function less than $q$, contradicting the statement of Theorem \ref{thm:knownresult}.\myqed
\end{proof}

%

We consider a partition on the columns of $DSC(q)$ into two blocks: $\mathcal{J} = \{J_1, J_1\}$ where $J_1$ corresponding to variables $x$ and $J_2$ corresponding to variables $y$. To complete the proof of the theorem it is sufficient to prove that the super sparse closure optimal value $z^{S.S.}$ of $DSC(q)$ is equal to optimal LP value $z^{LP}$ of $DSC(q)$. 

\begin{lemma}
	$z^{S.S} = z^{LP}$ for $DSC(q)$.
\end{lemma}

\begin{proof}
	Let $P$ be the integer set for $DSC(q)$. 
	Since $P^{S.S} = P^{(J_1)} \cap P^{(J_2)}$, Observation \ref{prop:fund} gives that $(\bar{x}, \bar{y}) \in P^{S.S}$ iff $(\bar{x}, \bar{y}) \in P^{LP}$ and $\bar{x} \in P^I|_{J_1}$ and $\bar{y}  \in P^I|_{J_2}$. But since $P$ is of covering-type and $SSC(q)$ is feasible, we have that $P^I|_{J_i} = [0,1]^{2^q}$, and thus $(\bar{x}, \bar{y}) \in P^{S.S.}$ iff $(\bar{x}, \bar{y}) \in P^{LP}$. This concludes the proof.\myqed
\end{proof}


\subsubsection{Proof of Theorem~\ref{thrm:covering}}

	Consider a covering problem $(\textup{C})$.	As in the packing case, there is an identification of sets of nodes of $\Gc$ with sets of indices of variables (the ``indices in the union of their support''), namely if $\I = \{I_1, I_2, \ldots, I_q\}$ is the given row index partition and the nodes of $\Gc$ are $\{v_1, v_2, \ldots, v_q\}$, then the set of vertices $\{v_i\}_{i \in I}$ corresponds to the indices $\bigcup_{i \in I} \bigcup_{r \in I_i} \supp(A_r) \subseteq [n]$. We will make use of this correspondence, and in order to make statements precise we use the function $\usupp : 2^{V(\Gc)} \rightarrow 2^{[n]}$ to denote this correspondence; with slight abuse of notation, for a singleton set $\{v\}$ we use $\usupp(v)$ instead of $\usupp(\{v\})$.
	
	Given a set of vertices $S \subseteq V(\Gc)$, let $x^{(S)}$ be the optimal solution of the covering problem projected to the variables relative to $S$, namely $x^{(S)} \in \argmin\{ (c|_{\usupp(S)})^T y ~\mid~ y \in P^I|_{\usupp(S)}\}$. Also, let $\zero^S(x^{(S)}) \in \R^n$ 
denote the solution appended by zeros in the original space, namely $\zero^S(x^{(S)})_i = x^{(S)}_i$ if $i \in \usupp(S)$ and $\zero^S(x^{(S)})_i = 0$ if $i \notin \usupp(S)$.
	
	Notice the following important property of $\zero^S(x^{(S)})$ (denote $S = \{v_i\}_{i \in I}$): for any row $r \in \bigcup_{i \in I} I_i$, since the support of $A_r$ is contained in $\usupp(S)$, the constraint $A_r x \ge b_r$ is valid for $P^I|_{\usupp(S)}$; therefore $A_r \zero^S(x^{(S)}) = A_r|_{\usupp(S)} x^{(S)} \ge b_r$. This gives the following.

	\begin{observation} \label{obs:feasCov}
		For any subset $S = \{v_i\}_{i \in I}$ of nodes of $\Gc$ and any row $r \in \bigcup_{i \in I} I_i$, $$A_r \zero^S(x^{(S)}) \ge b_r.$$
	\end{observation}

	We start by showing that the solutions $x^{(M)}$, for $M$ in a mixed stable set $\M$, can be used to provide a lower bound on the optimal value of $P^{\V,C}$.

\begin{lemma}\label{lem:addstabletsetcover}
Let $\mathcal{M}$ be a mixed stable set for $\Gc$ subordinate to $\V$. Then $$z^{\V,C} \ge \sum_{M \in \mathcal{M}} (c|_{\usupp(M)})^T x^{(M)}.$$ 
\end{lemma}

\begin{proof}
Consider an optimal solution $x^* \in \argmin\{c^Tx \,|\, x \in P^{\mathcal{V}, C}\}$ of the row block-sparse closure. Since 
$P^{\mathcal{V}, C} = \bigcap_{S \in \mathcal{V}} P^{(S)}$, Observation \ref{prop:fund} implies that $x^*|_{\usupp(S)} \in P^I|_{\usupp(S)}$ for all $S \in \V$. Moreover, since for every set $M$ in the mixed stable set $\M$ there is $S \in \V$ containing $M$, this implies that $x^*|_{\usupp(M)} \in P^I|_{\usupp(M)}$ for all $M \in \M$. Then by the optimality of $x^{(M)}$, we get $(c|_{\usupp(M)})^T (x^*|_{\usupp(M)}) \ge (c|_{\usupp(M)})^T x^{(M)}$ for all $M \in \M$.

	Then we can decompose the optimal solution $x^*$ based on the variables $\usupp(M)$ and use the non-negativity of $c$:
	\begin{align*}
		z^{\mathcal{V}, C} &= c^T x^* = \sum_{M\in \mathcal{M}} (c|_{\usupp(M)})^T (x^*|_{\usupp(M)}) + \sum_{i \notin \bigcup_{M \in \M} \usupp(M)} c_i x^*_i \ge \sum_{M\in \mathcal{M}} (c|_{\usupp(M)})^T x^{(M)},	
	\end{align*}
	where the first equality uses the fact that if $M_1, M_2 \in \mathcal{M}$, then $\usupp(M_1)\cap \usupp(M_2) = \emptyset$. This concludes the proof. \myqed
\end{proof}


	Now show how to put solutions $x^{(M)}$ together to get a feasible solution for the covering problem, thus providing an upper bound on $z^I$. Recall the definition of mixed chromatic number $\bar{\eta} = \bar{\eta}^\V(\Gc)$ and consider covering mixed stable sets $\M_1, \ldots, \M_{\bar{\eta}}$ (i.e., $V(\Gc) = \bigcup_i \bigcup_{M \in \M_i} M$).
 
	Define $u \in \R^n$ as the pointwise maximum of the solutions $\{\zero^M(x^{(M)}))\}_{i, M \in \M_i}$. Since the matrix $A$ in the problem is non-negative, Observation \ref{obs:feasCov} implies that $u$ is a feasible solution for the covering problem $(\textup{C})$. Thus, using the non-negativity of $c$ and of the $\zero^M(x^{(M)})$'s:
	\begin{align*}
		z^I &\leq c^T u \le \sum_{i, M \in \M_i} c^T \zero^M(x^{(M)}) = \sum_{i, M \in \M_i} (c|_{\usupp(M)})^T x^{(M)} \le \sum_i z^{\V,C} = \bar{\eta} \cdot z^{\V,C},
	\end{align*}
where the first inequality follows from definition of $z^I$ and feasibility of $u$, the second inequality follows from non-negativity of $c$, 	and the last inequality follows from Lemma \ref{lem:addstabletsetcover}. This concludes the proof of Theorem \ref{thrm:covering}. 	
	
%

\subsubsection{Proof of Theorem~\ref{thm:stoccovertight}}


Now we prove Theorem~\ref{thm:stoccovertight} by constructing a covering instance. Since the construction is quite involved, we start with an example.


\paragraph{Example of the construction.} We exemplify the construction for $K = 2$ and with a worse gap, and then we generalize/strengthen it (the discussion here will be somewhat informal). In this case the covering IP is the following (notice the indices of the $x$ variables in the different constraints):
	\begin{align}
		\min & \sum_{i} x_i + \infty \cdot (y_1 + y_2)\\
		s.t.~& \left[\left(\begin{array}{c} 1\\1\\0\\0 \end{array}  \right) x_1 + \left(\begin{array}{c} 0\\0\\1\\1 \end{array}  \right) x_2  \right] + \left[\left(\begin{array}{c} 1\\0\\1\\0 \end{array}  \right) x_3 + \left(\begin{array}{c} 0\\1\\0\\1 \end{array}  \right) x_4  \right] + \ones \cdot y_1 \ge \ones \label{eq:covLB1}\\
		& \left[\left(\begin{array}{c} 1\\1\\0\\0 \end{array}  \right) x_1 + \left(\begin{array}{c} 0\\0\\1\\1 \end{array}  \right) x_3  \right] + \left[\left(\begin{array}{c} 1\\0\\1\\0 \end{array}  \right) x_2 + \left(\begin{array}{c} 0\\1\\0\\1 \end{array}  \right) x_4  \right] + \ones \cdot y_2 \ge \ones \label{eq:covLB2}\\
		&x \in \Z_+^4, ~~y \in \Z^2_+. 
	\end{align}
	We will use the partition of rows $\I = \{I_1, I_2\}$, where $I_1 = \{1, 2, 3, 4\}$ (so corresponds to the first sets of covering constraints) and $I_2 = \{5, 6, 7, 8\}$.
	
	The only (minimal) ways to satisfy the first set of constraints is to set either $x_1=x_2=1$ (and all else to 0), or $x_3=x_4=1$ (and all else to 0), or $y_1=1$ (and all else to 0); because of the cost of the $y$ variables, actually we will always have $y=0$ in an optimal solution. To satisfy the second set of constraints the situation is similar, but the indices on the $x$ variables are permuted so that we need $x_1=x_3=1$ or $x_2=x_4=1$. So the best way to satisfy both of the constraints \textbf{simultaneously} is to set almost all $x$ variables to 1 (actually we can just set $x_1=x_2=x_3=1$). This gives cost of 3 for the IP.

	Now consider optimizing over the weak specific-scenario cuts closure $P^{{\V,C}}$ (where the row support list is $\V = \{ \{v_1\}, \{v_2\}\}$), i.e., the closure corresponding to the cuts on $(x,y_1)$ variables and on $(x, y_2)$ variables. Since the $y_i$ variable can be used to satisfy the $i$th set of covering constraints, it is easy to see that the only undominated $(x,y_1)$-cuts are the ones implied only the first set of covering constraints (\ref{eq:covLB1}), and similarly the only undominated $(x,y_2)$-cuts are the ones implied only by the second set of covering constraints (\ref{eq:covLB2}). Thus, the point  $x_1=x_2=x_3=x_4=\frac{1}{2}$, $y_1=y_2=0$ belongs to $P^{{\V,C}}$, giving $z^{*}= z^{\V, C} \le 2$.
	
	Together, these observations give that $\frac{z^{I}}{z^{\V, C}} \ge 3/2$.


\paragraph{General construction.}

	We start with the special set system that is used to define the columns of the covering program.
	
\begin{lemma}\label{lemma:planesPart}
Let $n \in \mathbb{Z}_{++}$. There is a collection $\G^1, \G^2, \ldots, \G^n$ with the following properties:

\begin{enumerate}
\item For each $i \in [n]$, $\G^i$ is a partition of $[n^n]$ and each set $G \in \G^i$ has size $n^{n - 1}$.

\item For any selection $G^1 \in \G^1, G^2 \in \G^2, \ldots, G^n \in \G^n$,
the intersection $\bigcap_{i=1}^n G^i$ is non-empty.
\end{enumerate}
\end{lemma}
\begin{proof}
Since $\left|[n^n]\right| = \left|[n]^n\right|$, let $g: [n]^n \rightarrow [n^n]$ be any bijection between the two sets. For $j \in [n]$, define the set $G^i_j  = \{g(u)\,|\,u \in [n]^n, u_i = j\}$. Define $\G^i = \{G^i_j \,|\,j\in [n]\}$. It is easy to check the following properties:
\begin{enumerate}
\item Given $i$, for any $j \in [n]$, $\left|G^i_j\right| = n^{n-1}$ and $\bigcup_{j=1}^n G^i_j = \left\{g(u)\,|\,u \in [n]^n \right\} = [n^n]$.
\item For a selection $G^1_{j_1} \in \G^1, G^2_{j_2} \in \G^2, \ldots, G^n_{j_n} \in \G^n$, consider $u = (j_1, j_2, \ldots, j_n)$. Then according to the definition, $g(u) \in \bigcap_{i=1}^n G^i_{j_i}$, so the intersection of these sets is non-empty.
\end{enumerate}
This concludes the proof. \myqed
\end{proof}

\begin{lemma} \label{lemma:coveringCover}
Let $n \in \mathbb{Z}_{++}$. Consider a collection $\G^1, \ldots, \G^n$ satisfying the properties of Lemma \ref{lemma:planesPart}, and consider $\bar{\G}^i \subseteq \G^i$ for $i = 1, \ldots, n$. If the sets in $\bigcup_{i = 1}^n \bar{\G}^i$ cover the whole of $[n^n]$, then there is $i \in [n]$ such that $\bar{\G}^i =
\G^i$.
\end{lemma}

\begin{proof}
By contradiction, suppose there is $G^1_{j_1} \in \G^1 \setminus
\bar{\G}^1, \ldots, G^n_{j_n} \in \G^n \setminus \bar{\G}^n$. Then
part 2 of Lemma \ref{lemma:planesPart}, there exists an element $u \in \bigcup_{i=1}^n G^i_{j_i}$, and since the sets in $\G^i$ partition
$[n^n]$ this means that $u$ is not covered by sets in $\bar{\G}^i$,
for all $i$; then $\bigcup_{i =1}^n \bar{\G}^i$ does not cover
$[n^n]$, a contradiction.\myqed
\end{proof}

	Now pick a prime number $n \ge \max\{K, 2\}$. We will construct an instance with variables $x_1, \ldots, x_{n^2}$ and $y_1, \ldots, y_K$, and each row-block will have $n^n$ constraints. Let $\G^1, \ldots, \G^n$ be a collection satisfying the properties from Lemma \ref{lemma:planesPart}, and let $\F_1, \ldots, \F_n$ be an affine $n$-design (we remind the readers -- each $\F_i$ partitions $[n^2]$ with $n$-subsets, and for $A \in \F_i, B \in \F_j$ with $i \neq j$, $|A \cap B| \le 1$); this affine design will be used to ``permute'' the indices of the $x$ variables from one set of covering constraints to the next (see example above). We consider the explicit enumeration $\F_i = \{F_i^1, \ldots, F_i^n\}$.
	
	For $k \in [K]$, we define the set
	\begin{align*}
		P_k = \left\{(x,y) \in \{0,1\}^{n^2 + K} \bigmid \sum_{i = 1}^n \sum_{j \in F_i^k} A^k_j x_j + \ones \cdot y_k \ge \ones\right\},
	\end{align*}
	where the set of vectors $\{A^k_j\}_{j \in F^k_i}$ is equal to the set of vectors $\{\chi_{G^i_{j'}}\}_{j' \in [n]}$; we note that it is not important which $A^k_j$ is assigned to which $\chi_{G^i_{j'}}$. 
	
	Then the covering integer program we consider is the following:
	\begin{align*}
		\min ~~& \sum_{j=1}^{n^2} x_j + n^n\cdot \sum_{k = 1}^K y_k\\
		s.t. ~~& \sum_{i = 1}^n \sum_{j \in F_i^k} A^k_j x_j + \ones \cdot y_k \ge \ones ~~~~~\forall k \in [K]\\
					&(x,y) \in \{0,1\}^{n^2 + K}. 
	\end{align*}
	(or equivalently $(x,y) \in \bigcap_{k \in [K]} P_k$). Let $P$ denote the set of solutions for this problem.
	
	Notice that this program has $K$ sets of $n^n$ covering constraints. We consider the partition of the covering constraints $\I = \{I_1, \ldots, I_K\}$ where $I_k = \{(k-1) n^n + 1, \ldots, kn^n\}$ (so $I_k$ corresponds to $P_k$). It is easy to see that the covering interaction graph $\Gc$ is a clique. 
	
	For $\V = \{\{v_1\}, \dots, \{v_K\}\}$, remember that $z^{*} = z^{\V, C}$. Therefore, we want to show that $\frac{z^I}{z^{\V, C}} \ge K - \epsilon$ (for sufficiently large $n$).
	We start by analyzing each $P_k$.
	
	\begin{lemma} \label{lemma:charPk}
		A vector $(\bar{x}, 0) \in \{0,1\}^{n^2} \times \{0,1\}^{K}$ belongs to $P_k$ if and only if there is $F_i^k$ such that $\bar{x}_j = 1$ for all $j\in F_i^k$.
	\end{lemma}
	
	\begin{proof}
		\noindent $(\Rightarrow)$ 
Let $\bar{\G}^i$ be the subset of the sets in $\G^i$ picked by
$\bar{x}$, namely $G^i_t$ belongs to $\bar{\G}^i$ iff
$\chi_{G^i_{j'}} = A^k_j$ for some $j$ with $\bar{x}_j = 1$. Since $\bar{y} = 0$, the fact that $(\bar{x}, \bar{y})$ belongs to $P_k$ implies that the sets in $\bigcup_{i = 1}^n \bar{\G}^i$ must cover
the whole of $[n^n]$. Lemma \ref{lemma:coveringCover} then
implies that there is one $\bar{\G}^i$ that equals $\G^i$, which
translates to having $\bar{x}_j = 1$ for all $j \in F^k_i$. 

		\smallskip \noindent $(\Leftarrow)$ This follows from the fact that the sets in $\G^i$ cover the whole of $[n^n]$.\myqed
	\end{proof}
	
	We can use this to lower bound the optimal value $z^I$ of the covering program $P$.

	\begin{lemma} \label{lemma:LBcovIP}
			$z^I \ge Kn - K^2$.
	\end{lemma}
	
	\begin{proof}
		First, we claim that if $(\bar{x}, 0) \in P$, then $\sum_{j \in [n^2]} \bar{x}_j \ge K n - K^2$. Let $S \subseteq [n^2]$ be the support of $\bar{x}$, so it is equivalent to show $|S| \ge n K - K^2$. Since $(\bar{x}, 0) \in P = \bigcap_{k=1}^K P_k$, using Lemma \ref{lemma:charPk} we have that for every
$k \in [K]$ there is $i(k)$ such that $S$ contains $F^k_{i(k)}$, so
$S \supseteq \bigcup_{k=1}^K F^k_{i(k)}$. By the inclusion-exclusion
principle, we have that $|S| \ge \left|\bigcup_{k=1}^K
F^k_{i(k)}\right| \ge \sum_{k =1}^K |F^k_{i(k)}| - \sum_{k\neq k'}
|F^k_{i(k)} \cap F^{k'}_{i(k')}|$. Using the definition of an
affine $n$-design, get the lower bound $|S| \ge nK - K(K - 1) \geq nK - K^2$.

	Now consider any solution $(\bar{x}, \bar{y}) \in P$. If $\bar{y} = 0$, we have just shown that this solution has value at least $Kn - K^2$; if $\bar{y} \neq 0$, this solution has value at least $n^n > Kn - K^2$. This concludes the proof. \myqed
	\end{proof}

	Finally we upper bound the optimal value of the $z^{\V, C}$.

	\begin{lemma} \label{lemma:UBcovSS}
		$z^{\V, C} \le n$.
	\end{lemma}
	
	\begin{proof}
		It suffices to show that the point $(\bar{x},\bar{y}) = \left(\frac{1}{n} \ones, 0\right)\in P^{\V, C}$. Recall that $P^{\V, C} = \bigcap_{k \in [K]} P^{(\{v_k\})}$, so we show $(\bar{x},\bar{y})$ belongs to all $P^{(\{v_k\})}$'s.
		Note that $(\bar{x},\bar{y})$ satisfies the linear programming relaxation; therefore, using Observation \ref{prop:fund}, to show that $(\bar{x}, \bar{y})$ belongs to $P^{(\{v_k\})}$ it suffices to prove that $(\bar{x}, \bar{y}_k) \in P^I|_{(x,y_k)}$, where we use $P^I|_{(x,y_k)}$ to denote the projection onto the variables $(x,y_k)$.
		
		Consider the following points $(x^u, y)$, for $u \in [n]$, constructed as:
\begin{eqnarray*}
y_k &=& 0, \\
y_{k'} &=& 1 \ \forall \ k' \in [K] \setminus \{k\},\\
x^u_j &=& \left \{ \begin{array}{cc} 1 & \textup{ if } j  \in F^k_u\\ 
0 & \textup{ otherwise.}
\end{array} 
\right. 
\end{eqnarray*}
It is straightforward to verify that $(x^u, y) \in P$ for $u \in [n]$. Thus, the average $ \frac{1}{n} \sum_{u \in [n]} (x^u, y) $ belongs to $P^I$. It then follows that $(\bar{x}, \bar{y}_k) = (\frac{1}{n} \ones, 0)|_{(x,y_k)}$ belongs to $P^I|_{(x,y_k)}$. This concludes the proof.\myqed
	\end{proof}
 
 Putting Lemmas \ref{lemma:LBcovIP} and \ref{lemma:UBcovSS} together we get $\frac{z^I}{z^{S.S.}} \ge \frac{Kn - K^2}{n} = K - \frac{K^2}{n}$. For large enough $n$, we get $\frac{z^I}{z^{\V, C}} 
 \ge K - \epsilon$. This concludes the proof of Theorem \ref{thm:stoccovertight}.


\subsection{Proof for packing-type problem with arbitrary $A$ matrix} \label{sec:proofpackinggen}


\subsubsection{Proof of Theorem \ref{thm:supersparseNS}}
	
	In this section, we use the same notation as that used in Section \ref{sec:packingMain}. So, let $\J$ be a partition of the index set of columns of $A$ (that is [n]). Let $V = \{v_1, \ldots, v_q\}$ be the vertices of $\Gp$ (based on Definition \ref{defn:packgraph}).
Let $\tilde{\V} = \{V^{u_1}, V^{u_2}, \dots, V^{u_k}\} \subseteq \mathcal{V}$ be the subset of sparse cut support list corresponding to the definition of corrected average density $D_{\V}$ (see Definition \ref{defn:corraverden}), so $V = \bigcup_{i = 1}^k V^{u_i}$ and $\frac{1}{k}\sum_{i = 1}^k |V^{u_i}| = D_{\V}$.

	Recall from Section \ref{sec:packingMain} a couple of definitions: first, the function $\phi$ maps subsets of vertices of $V$ to the corresponding variable indices, namely if $S = \{v_i\}_{i \in I}$ then $\phi(S) = \bigcup_{i \in I} J_i$. 
	
	
	For the purpose of this section, let $x^{(S)} = \argmax\left\{ (c|_{\phi(S)})^T (x|_{\phi(S)}) \mid x \in P^I\right\}$. (Notice that this is different from the definition used in Section \ref{sec:packingMain}.)
	
		\begin{lemma} \label{lemma:genPacking}
		For any set $\tilde{V} \in \tilde{\V}$ we have $z^{\mathcal{V}, P} \leq c^T x^{(\tilde{V})} + \sum_{v \in V \setminus \tilde{V}} c^T x^{(v)}$.
	\end{lemma}
	
	\begin{proof}
 Fix any $\tilde{V} \in \tilde{\V}$ and $S \subseteq \tilde{V}$. Let $x^* = \argmax \{ c^T x \,|\, x \in P^{\mathcal{V}, P}\}$ be an optimal solution corresponding to the optimization over $P^{\mathcal{V}, P}$. Since $P^{\V,C} = \bigcap_{\tilde{V}' \in \tilde{\V}} P^{(\tilde{V}')}$, we have that $x^* \in P^{(S)} \supseteq P^{(\tilde{V}')}$. From Observation \ref{prop:fund} we then get $x^*|_{\phi(S)} \in P^I|_{\phi(S)}$. 
 
Thus we get $(c|_{\phi(S)})^T (x^*|_{\phi(S)}) \le (c|_{\phi(S)})^T (x^{(S)}|_{\phi(S)}) \leq c^T x^{(S)}$, where the first inequality follows from optimality of $x^{(S)}$ and the second inequality follows from non-negativity of $c$ and $x^{(S)}$. 

In particular, since $\tilde{\V}$ covers $V$, we can apply this to the any singleton $S = \{v\}$ and get $(c|_{\phi(v)})^T (x^*|_{\phi(v)}) \le c^T x^{(v)}$.
 
 Applying this bound, we obtain that for any $\tilde{V} \in \V$
	\begin{align*}
		z^{\mathcal{V}, P} &= c^T x^* = (c|_{\phi(\tilde{V})})^T (x^*|_{\phi(\tilde{V})}) + \sum_{v \in V \setminus \tilde{V}} (c|_{\phi(v)})^T (x^*|_{\phi(v)}) \le c^T x^{(\tilde{V})} + \sum_{v \in V \setminus \tilde{V}} c^T x^{(v)}.
	\end{align*}
	This concludes the proof. \myqed
	\end{proof}
	
Now we are ready to complete the proof of the theorem. Using Lemma \ref{lemma:genPacking} for all sets in $\tilde{\V}$ and adding up these inequalities we obtain that
	\begin{align}
		k \cdot z^{\mathcal{V}, P} &\leq \sum_{i = 1}^k c^T x^{(V^{u_i})} + \sum_{i = 1}^k \left(\sum_{v \in V \setminus V^{u_i}} c^T x^{(v)} \right) \notag \\
		\label{eq:oneside} &= \sum_{i = 1}^k c^T x^{(V^{u_i})} + \sum_{v \in V} \miss(v) \cdot c^T x^{(v)},
	\end{align}
	where $\miss(v) = |\{i \in [k]\,|\, v \not \in V^{u_i}\}|$, that is the number of sparse-cut types in $\V$ in which the variables corresponding to vertex $v$ do not appear.
	
	Moreover it follows from the definition of $x^{(S)}$ that $x^{S} \in P^I$ and therefore we have that $z^I \ge c^T x^{(S)}$ for every subset $S \subseteq V$. Thus, we obtain that
	\begin{align*}
		z^I &\geq \max \left\{ \max_{i \in [k]} \{ c^T x^{(V^{u_i})} \} ~,~ \max_{v \in V} \{ c^T x^{(v)} \} \right\} \\
		&\geq \frac{1}{k + \sum_{v \in V} \miss(v)} \left(\sum_{i = 1}^k c^T x^{(V^{u_i})} + \sum_{v \in V} \miss(v) \cdot c^T x^{(v)} \right) \\
		&\geq \frac{k}{k + \sum_{v \in V} \miss(v)} \cdot z^{\mathcal{V}, P}.
	\end{align*}
where the second inequality follows from taking a weighted average, the third inequality follows from~\eqref{eq:oneside}. Finally, the next lemma shows that $k + \sum_v \miss(v) = k + kq - kD_\V$, concluding the proof of the theorem. 

	\begin{lemma}
		$kD_{\mathcal{V}} + \sum_{v \in V} \miss(v) = kq$.
	\end{lemma}
	
	\begin{proof}
		We perform a simple double counting. Consider the $V/\{V^i\}_i$ incidence matrix $B \in \{0,1\}^{k \times q}$ defined as $B_{i,v} = 1$ if $v \in V^i$ and $B_{i,v} = 0$ if $v \notin V^i$. Using the definition of $D_{\mathcal{V}}$ we have:
	\begin{eqnarray}
		\label{eq:nonzero} k D_{\mathcal{V}} = \left| \{ (i,v) \in [k]\times V\,|\, B_{i,v} \neq 0\}\right|.
	\end{eqnarray} 
On the other hand, from the definition of $\miss(v)$ we have that
	\begin{align}
		\sum_{v \in V} \miss(v) &= \sum_{v \in V} |\{i \in [k]\,|\, v \not \in V^i\}| = \left| \{ (i,v) \in [k]\times V\,|\, B_{i,v} = 0\}\right|. \label{eq:zero}
	\end{align}
By (\ref{eq:nonzero}) and (\ref{eq:zero}), we have that $kD_{\mathcal{V}} + \sum_{v \in V} \miss(v) = kq$. This concludes the proof.\myqed
	\end{proof}

\subsubsection{Proof of Theorem~\ref{thm:SSarbitAstartight}}

	We consider the following integer program with $2 K  - 1$ variables:
\begin{align}
\textup{max}~~& x_{K} + \sum_{j = K + 1}^{2K - 1} x_{j} \notag\\
s.t.~~&  \sum_{i = 1}^{K} x_i = 1 \label{eq:firstIPcons}\\
		  & x_i + x_j \leq 2 - \epsilon  ~~~~~~\forall i \in \{1, \dots, K - 1\}, ~~~\forall j \in \{K + 1, \dots, 2K-1\}\setminus \{K + i\} \label{eq:othercons}\\
			& x_{K} + x_j \leq 2 - \epsilon  ~~~~~\forall j \in \{K + 1, \dots, 2K -1\} \label{eq:lastIPcons}\\
			& x \in \{0,1\}^{2K - 1}.\notag
\end{align} 
 (We assume $\epsilon < \frac{K-1}{K}$.) Let $P$ denote the integer set relative to this problem.
 
 We consider the partition $\J = \{J_1, \ldots, J_{K}\}$ of the columns given by $J_1 = \{1, \dots, K\}$, $J_i = \{K + i - 1\}$ for $i \in 2, \dots, K$. Notice that the packing interaction graph $\Gp$ for this program is a star on $K$ nodes. Writing explicitly $\Gp = (V,E)$ with $V = \{v_1, \dots, v_{K}\}$ and $$E= \{ (v_1, v_2), (v_1, v_3), (v_1, v_4) , \dots, (v_1, v_{K}) \}.$$ Since we are in the context of the super sparse closure, we have the support list $$\mathcal{V} = \left\{\{v_1\}, \{v_2\}, \dots, \{v_{K}\}\right\}.$$ 

	We show the bound $z^{S.S.} \ge K \cdot z^I - \epsilon$, and start by lower bounding $z^{S.S.}$.
	
	\begin{lemma}
		$z^{S.S.} \geq K  - \epsilon$
	\end{lemma}
	
	\begin{proof}
	We claim that the point $\bar{x}$ given by $\bar{x}_j = \frac{\epsilon}{K - 1}$ for all $j = \{1, \dots, K - 1\}$, $\bar{x}_{K} = 1 - \epsilon$, $\bar{x}_{j} = 1$ for all $j \in \{K + 1, \dots, 2K - 1\}$ belongs to the natural sparse closure $P^{S.S.}$, proving it using Observation~\ref{prop:fund}.
	
	First, it is easy to check that $\bar{x}$ belongs to the LP relaxation $P^{LP}$. Moreover, note that $P^I|_{J_1} = \{(x_1, \dots, x_{K}) \in [0,1]^K \,|\, \sum_{i = 1}^{K} x_i = 1\}$, 
and hence $\bar{x}|_{J_1} \in P^I|_{J_1}$. In addition, $P^I|_{J_j} = [0, \ 1]$ for $j \in \{2, \dots, |V|\}$, and thus $\bar{x}|_{J_j} = \bar{x}_{K + j-1} \in P^I|_{J_j}$. Since $P^{S.S.} = \bigcap_{j = 1}^K P^{(J_j)}$, from Observation \ref{prop:fund} we get that $\bar{x} \in P^{S.S.}$.\myqed
	\end{proof}

	To complete the proof, we show that the optimal value of the IP is (at most) $1$, namely exactly one of the variables $x_{K}, x_{K + 1}, \dots, x_{2K - 1}$ can take a value of $1$ and the others are zero. So consider any feasible solution $\bar{x} \in \{0,1\}^{2K - 1}$. If $\bar{x}_{K}=1$, then the constraints \eqref{eq:lastIPcons} imply that $x_{K +1} = x_{K+2} = \dots = x_{2K - 1} = 0$. On the other hand if $\bar{x}_K = 0$, then by constraint \eqref{eq:firstIPcons} there is some $i \in [K - 1]$ with $\bar{x}_i = 1$, and so constraints \eqref{eq:othercons} imply $\bar{x}_j = 0$ for all $j \in \{|V| + 1, \dots, 2K + 1\} \setminus \{K + i\}$, and so at most $\bar{x}_j$ can take value 0. 
	
	Since $z^I \le 1$ and $z^{S.S.} \ge K - \epsilon$, we get the desired bound $z^{S.S.} \ge K \cdot z^I - \epsilon$, concluding the proof of Theorem \ref{thm:SSarbitAstartight}.


\subsubsection{Proof of Theorem~\ref{thm:arbitANStight}}

We will construct an example with $2K$ variables. To start, for $k \in [K]$ let $P^I_k$ be the convex hull of the points 
	\begin{align*}
		P_k := &\left\{ (x, y) \in \{0, 1\}^{K + K} \bigmid \right. \\
		&\left. y_k = 1 \textrm{ if and only if either $[x_k = 1, x_i = 0 ~\forall i \neq k]$ or $[x_k = 0, x_i = 1 ~\forall i \neq k]$} \right\}
	\end{align*}
	We then consider the integer program
	\begin{align*}
		\max ~~& \sum_{k = 1}^K y_k \\ 
		\textup{s.t.} ~~& (x,y) \in \bigcap_{k \in [K]} P^I_k \cap \{0,1\}^{2K}.
	\end{align*}
	Let $P$ denote the associated integer set. 
The partition of variable indices we consider is $\J = \{J_0, J_1, J_3, \dots, J_K\}$, where $J_0$ corresponds to the variables $x$, and each $J_k$ corresponds to variable $y_k$ for $k \in [K]$. Notice that the packing interaction graph $\Gp$ is a star on $K+1$ nodes; explicitly, $\Gp = (V,E)$ with $V= \{v_0, \ldots, v_K\}$ and $E = \{\{v_0,v_1\}, \ldots, \{v_0, v_K\}\}$. Recall we are in the natural sparse closure setting, so the support list $\V$ in this case equals the edge set $E$.
	
	We show that $z^{N.S.} \ge K \cdot z^I$. For that, we start by lower bounding $z^{N.S.}$. 

	\begin{lemma}
		$z^{N.S.} \geq K$.
	\end{lemma}
	
	\begin{proof}
		We show that the solution $(\bar{x}, \bar{y})$ given by $\bar{x} = \frac{1}{2} \ones$ and $\bar{y} = \ones$ belongs to $P^{N.S.}$. 	Following Observation \ref{prop:fund}, to show $(\bar{x},\bar{y}) \in P^{N.S.}$ it suffices to show $(\bar{x},\bar{y}) \in P^{LP}$ and $(\bar{x}, \bar{y}_k) \in P^I|_{(x,y_k)}$ for all $k \in [K]$. Notice that $P^{LP} = \bigcap_{k \in [K]} P^I_k$ and $P^I|_{(x,y_k)} = P^I_k|_{(x,y_k)}$ (the latter uses the fact $P^I_j|_{(x,y_k)} = [0,1]^{K+1}$ for $j \neq k$). Thus it suffices to show $(\bar{x}, \bar{y}) \in P^I_k$ for all $k \in [K]$
	
		 For that, fix $k \in [K]$ and consider the points $(x^{k1}, e^k)$ and $(x^{k2}, e^k)$, where $e^i$ is the $i$th canonical basis vector and 
	\begin{align*}
		&x^{k1}_i = \left\{\begin{array}{rl} 1& \textup{if } i = k,\\
		 0 & \textup{otherwise}
		\end{array}\right.\\
		&x^{k2}_i = \left\{\begin{array}{rl} 0& \textup{if } i = k,\\
		 1 & \textup{otherwise}
		\end{array}\right. .
	\end{align*}
	By definition both these points belong to $P_k$; the average $\frac{1}{2} (x^{k1}, e^k) + \frac{1}{2} (x^{k2}, e^k) = (\frac{1}{2} \ones, e^k)$ also belongs to $P_k$. Moreover, since the constraints defining $P^I_k$ are independent of variable $y_i$ for $i \neq k$, we have that $(\frac{1}{2} \ones, \ones) = (\bar{x}, \bar{y})$ also belongs to $P^I_k$. This concludes the proof. \myqed
	\end{proof}

	Now it is easy to see from the definition of $P_k$ that no feasible solution to the IP can set more than one $y$ variable to 1, and hence the optimal value $z^i$ is at most 1. Together with the previous lemma, this gives the desired bound $z^{N.S.} \ge K \cdot z^I$, thus concluding the proof of the theorem. 
	

\section{Acknowledgements}
We would like to thank Carla Michini for her comments that helped improve the presentation in this paper. Santanu S. Dey and Qianyi Wang acknowledge the support from NSF CMMI Grant 1149400.

\ifmp
	\bibliographystyle{spmpsci}  
\else
	\bibliographystyle{plain}
\fi
\bibliography{test}

\appendix
\section{Upper bound on $z^{\textup{cut}}$}\label{AppendixCutAlgo}
Assume we have the general formulation
\begin{eqnarray*}
\textup{max}&&  c^Tx \notag \\
s.t. &&Ax \leq b \notag \\
&&x \in \mathbb{B}^n,
\end{eqnarray*}
where $A \in \R^{m\times n}$. Recall that we are interested in three
type of problems: packing, covering and packing with arbitrary
constraint matrix. All these categories will be written in the form of the 
formulation above with different restrictions on $A$ and $c$. Let
$N_i = \{j \in [n]\,|\, A_{ij} \neq 0\}$ be the index set of non-zero
entries of $i^{th}$ row of $A$. Let $\mathcal{N} = \{N_1, N_2,\ldots, N_t\}$, denote $P^{\mathcal{N}} = \bigcap_{i = 1}^t P^{(N_i)}$ and $z^{cut} = \textup{max}_{\textup{$x \in P^{\mathcal{N}}$}}~ c^Tx $. 

Our basic strategy is the following: we keep adding cuts on the support of some
$N_i$ and checking whether the LP solution will improve. We stop adding cuts when the objective function value does not change, thus obtaining an upper bound on $z^{cut}$. The formal
algorithm is shown as Algorithm 1.

\begin{algorithm}[h!]\label{alg:1}
\caption{Estimating $z^{cut}$}
\begin{algorithmic}
\STATE \textbf{input:} {$P = \{x | Ax \leq b \}$,$\mathcal{N} = \{N_1,N_2,\ldots, N_t\}$, $z^{old} = -\infty$, $z^{new} = -\infty$, $\epsilon = 10^{-6}$}
\STATE  $i \leftarrow 1$, $count \leftarrow 0$ 
\LOOP
\STATE  Solve $x^* = \textup{argmax}_{x \in P} c^Tx$,$z^{new} = c^Tx^*$\;\\
\IF{$x^*$ is integral}
\STATE $z^{cut} = z^{new}$ 
\STATE Exit Loop.
\ELSIF{$z^{new} - z^{old} > \epsilon$}
\STATE $z^{old} \leftarrow z^{new}$
\STATE Generate a valid cut $\alpha x \leq \beta$ on the support of $N_i$ based on Algorithm 2
\STATE $P \leftarrow P \bigcap \{x| \alpha x \leq \beta\}$
\STATE  $count \leftarrow 0$
\ELSE 
\STATE $z^{old} \leftarrow z^{new}$
\IF{count = t}
\STATE $z^{cut} = z^{new}$
\STATE Exit Loop.
\ELSE
\STATE $i \leftarrow i+1 (mod \ t)$
\STATE $count \leftarrow count+1$
\ENDIF
\ENDIF
\STATE 
\ENDLOOP
\textbf{end loop}
\end{algorithmic}
\end{algorithm}

Once we stop adding cut on some $N_i$, we check whether there is a valid cut on $N_{i+1}$. The index $count$ is the number of groups of supports that adding cuts will not improve the optimal objective function value. Also, as long as adding a cut produces improvement on the objective value, $count$ will be reset as 0. The algorithm terminates when one of the following happens:
\begin{enumerate}
\item An integral feasible solution is found.
\item The parameter $count$ equals to the number of supports $t$.
\end{enumerate}

In the algorithm, we call a routine to generate the cut on $N_i$ that is formally shown as Algorithm 2. 
Assume that $\alpha^T x \leq \beta$ is a valid cut on $N_i$
for some $i$, then $\alpha^T \hat{x}  \leq \beta$ holds for all
$\hat{x}\in P^I$. However, as the formulation of $P^I$ is implicit,
we apply the technique of row generations. Let $X$ be a subset of
all integral points in $P^I$. At the beginning, $X = \emptyset$. And
we generate a valid cut $(\alpha^*, \beta^*)$ for $X$. Then we solve
the following IP
\begin{eqnarray*}
\textup{max} && \alpha^* x -\beta^* \\
s.t. && x \in P^I.
\end{eqnarray*}
If the optimal value is less or equal to 0 then it means the cut is
valid for $P^I$, otherwise let $X = X \bigcup \{x^*\}$, where $x^*$
is the optimal solution. By re-applying this process, we will
either obtain a valid cut or a certificate that no valid cut exists.

\begin{algorithm}[h!]\label{alg:2}
\caption{Cut generation on $N_i$}
\begin{algorithmic}
\STATE \textbf{Input:} {$P = \{x | Ax \leq b \}$, $P^I = conv\_hull\{x| x\in P, x\in Z^n\}$, $x^*$} 
\STATE $X \leftarrow \emptyset$, $\epsilon \leftarrow 10^{-6}$
\LOOP
\STATE Solve $(\alpha^*,\beta^*) = \textup{argmax}_{x^T \alpha \leq \beta,\forall x \in X,\|\alpha\|_1 = 1, \textup{support of } \alpha = N_i} ~x^{*T}\alpha - \beta$
\IF{$x^{*T}\alpha - \beta > \epsilon$}
\STATE Solve $x^0 = \textup{argmax}_{x \in P^I} ~\alpha^* x -\beta^* $
\IF{$\alpha^*x^0 -\beta^* > \epsilon$}
\STATE  $X \leftarrow X\bigcup \{x^0\}$
\ELSE
\STATE Return $(\alpha^*,\beta^*) $
\STATE Exit Loop
\ENDIF
\STATE Return $(\alpha,\beta) = (\vec{0},0)$
\STATE Exit Loop
\ENDIF
\STATE 
\ENDLOOP
\textbf{end loop}
\end{algorithmic}
\end{algorithm}

\end{document}